\documentclass[12pt]{article}

\usepackage{url}  

\usepackage{arxiv_package}

\usepackage{graphicx}
\usepackage{amssymb,mathtools}
\usepackage{multirow}
\usepackage{diagbox}

\DeclareRobustCommand{\qzz}[2]{#2}
\soulregister{\qzz}{2}

\newcommand{\anonymize}[1]{}

\hypersetup{
    colorlinks=true,
    citecolor = blue,
    linkcolor=blue
}

\usepackage{dsfont}
\usepackage{comment}

\usepackage{cleveref}
\usepackage[toc,page,titletoc]{appendix}

\usepackage{algorithm}
\usepackage{algpseudocode}

\title{Multiple conditional randomization tests for lagged and
  spillover treatment effects}

\author[]{Yao Zhang\vspace{-11pt}\thanks{We thank Jos Twisk and Ye Wang
    for sharing the trial data with us. We thank
    Stephen Bates, Jing Lei, and Paul Rosenbaum for reading and
    commenting on a previous version of this manuscript.} \\
    Department of Statistics, Stanford University  \vspace{2pt}
    \\ and  \\  Qingyuan Zhao\thanks{Partly supported by EPSRC grant   
      EP/V049968/1.} \\ Statistical Laboratory, University of Cambridge}

\begin{document}

\maketitle

\begin{abstract}
  We consider the problem of constructing multiple independent
conditional randomization tests using a single dataset. Because the
tests are independent, the randomization p-values can be interpreted
individually and combined using standard methods for multiple
testing. We give a simple, sequential construction of such
tests, and then discuss its application to three problems: Rosenbaum's
evidence factors for observational studies, lagged treatment effect in
stepped-wedge trials, and spillover effect in randomized trials with
interference. We compare the proposed approach with some existing
methods using simulated and real datasets. Finally, we establish a
more general sufficient condition for independent conditional
randomization tests.


\end{abstract}


\section{Introduction}\label{new_intro}
The growth of randomized experiments in social and biomedical sciences
has created numerous opportunities for the application of
randomization tests. These tests, dating back to
\citet{fisher1935design} and \citet{pitman1937significance}, are based
solely on the physical act of randomization, enabling them to yield
p-values that are exact in finite samples without relying on any
distributional assumptions. Variations of randomization tests have
also found wide applications in non-randomized settings in which
modelling the full data distribution is prohibitively difficult; see
\citet{zhangzhao2022} for a recent review.

Many modern randomized experiments have a complex design, and
it may be challenging to design simple randomization tests. This is a
particularly difficult problem when the
causal hypothesis is not strong enough to impute all possible
potential outcomes. There is much methodological progress in the last
decade on this problem by conditioning on a suitable statistic
(usually a function of the treatment). Some notable examples of {\it
  conditional randomization tests} (CRTs) include those for
meta-analysis \citep{zheng2008multi}, covariate adjustment
\citep{hennessy2016conditional}, and unit interference
\citep{athey2018exact,basse2019randomization,puelz2019graph}. Here,
conditioning can be viewed as a way to increase the
precision of hypothesis testing. The theory for a single CRT is briefly
reviewed in \Cref{sec:single-crt} below.

In practice, it is often of interest to conduct multiple CRTs to
increase power or assess different causal hypotheses. This is the
problem we will study in this paper. Before outlining our proposal, we
will give three examples where multiple CRTs could be useful.

\begin{example}[Evidence factors] \label{ex:evidence-factor}
  In a series of work,
  \citet{rosenbaum2010evidence,rosenbaum2017general,rosenbaum2021replication}
developed a concept called ``evidence factors'' for observational
studies, which are essentially pieces of evidence regarding
different aspects of an abstract causal hypothesis. Roughly speaking,
each evidence factor is characterized by a worst-case p-value in a
sensitivity analysis for a different (partially) sharp null
hypothesis. Rosenbaum's key observation is that in certain
settings, these evidence factors are ``nearly independent'' in the sense that the worst-case p-values stochastic dominate the
multivariate uniform distribution over the unit hypercube.
\end{example}

\begin{example}[Lagged treatment effect]\label{label:s_exp}
The stepped-wedge randomized controlled trial is a monotonic
cross-over design in which all units begin in the control group and subsequently cross over to the treatment group at staggered time
points \citep{hussey2007design}. This design is also known as
``staggered adoption'' in economics
\citep{abraham2018estimating,athey2018design}. It has
become a popular trial design in medical and policy research, especially in cases
where it is impractical to administer the treatment to all units at
the same time or where for ethical reasons the treatment should be
given to everyone since it has shown effectiveness already. Because
the
treatment is administered at different times, the stepped-wedge design
allows us to estimate time-lagged treatment effects. This is usually
done using mixed-effects models
\citep{hussey2007design,hemming2018reporting,li2021mixed}, but the
inference is not reliable when the models are misspecified
\citep{thompson2017bias,ji2017randomization}. Randomization tests have
been proposed as an alternative to test the global null hypothesis
that assumes no treatment effect whatsoever
\citep{ji2017randomization,wang2017use,thompson2018robust,hughes2020rob}. However,
to our knowledge, no randomization test for lagged treatment effects has
been developed in this context.
\end{example}

\begin{example}[Testing the range of spillover effect]\label{label:c_exp}
When experimental units are geographical regions or users
connected through a social network, it is often of interest to know if the effect of
a treatment can spill over to neighbouring units. For example,
\citet{jayachandran2017cash}  conducted a randomized controlled trial
to study if payments for ecosystem services can reduce deforestation
in Uganda.
After receiving these payments, forest owners were expected to engage
in forest conservation within their respective villages. However,
there was also a potential spillover effect, where these forest owners
may deforest more in nearby untreated villages.
 CRTs for spillover effects have been studied recently by, among
 others, \citet{aronow2012general}, \citet{athey2018exact},
 \citet{basse2019randomization}, and \citet{puelz2019graph}.
 However, a single CRT can only determine the presence of a spillover
 effect within a specified distance. Multiple CRTs can assess how far the treatment effect can spread so policymakers can determine, e.g., in the deforestation
 example, how to maximize forest conservation by allocating payments
 to selected villages.
\end{example}

A common challenge in these examples is that a naive application of
existing CRTs for different causal hypotheses (e.g.\ different time
lags or spill-over distances) will give dependent p-values, which
complicates the decision. In this paper, we will propose a general,
sequential construction of conditional randomization tests that
generate ``nearly independent'' p-values. These p-values can be viewed as
independent pieces of evidence for different causal hypotheses (or
as ``evidence factors'' in Rosenbaum's terminology), and can be combined
using standard multiple testing methods such as Fisher's and
Stouffer's methods for testing the intersection of
the causal hypotheses
\citep{fisher1925statistical,stouffer1949american}, or Hommel's method
and more general closed testing procedures for
controlling the family-wise error rate \citep{marcus1976closed}.
We will illustrate this sequential testing approach by applying it to
the problems in \Cref{ex:evidence-factor,label:s_exp,label:c_exp} and
use simulations to investigate the power of the proposed method. 
Some concluding remarks will be offered at the end of the article.

\section{Review: A single conditional randomization test}
\label{sec:single-crt}




Consider a randomized experiment on $N$ units. Let $\bm Z \in
\mathcal{Z}$ denote the treatment assignment that is randomized.
For every unit $i\in [N]:=\{1,\cdots,N \},$ we denote all its
\emph{potential outcomes} by $(Y_i(\bm z)\mid \bm z\in
\mathcal{Z})$. Let  $\bm Y(\bm z) = (Y_1(\bm z), \dotsc, Y_N(\bm z)).$
The collection of potential outcomes for all units and treatment
assignments is referred to as the \emph{potential outcomes schedule},
which is denoted by
$\bm W = (\bm Y (\bm z): \bm z \in
\mathcal{Z}) \in \mathcal{W}$.
After $\bm Z$ is randomized, we observe a single outcome  $Y_i$ for every
unit $i$.  We assume that all the observed outcomes $\bm Y =
(Y_1,\dotsc,Y_N)$ are consistent, meaning that they match the
potential outcomes for the realized assignment, i.e., $\bm Y = \bm Y(\bm Z)$.

Treatment effects such as lagged or spillover effects are often
defined as contrasts between potential outcomes. By randomizing $\bm
Z$, the experiment establishes independence between $\bm Z$ and $\bm
W$ that eliminates all confounding factors between the treatment and
the outcome. This heuristic is formalized by the assumption below.

\begin{assumption}[Randomized experiment] \label{assump:randomization}
  $\bm Z \independent \bm W$ and the density function $\pi(\cdot)$ of
  $\bm Z$ is known and positive everywhere in $\mathcal{Z}$.
\end{assumption}


In this paper, we are interested in testing causal
hypotheses under which some unknown entries of $\bm W$ can be imputed
from $\bm Y$. The simplest example is the global null:
for all assignments $\bm z,\bm z'\in \mathcal{Z}$, the potential
outcomes are identical, i.e., $\bm Y(\bm z) = \bm Y(\bm z')$. This is
also known as Fisher's exact null. Under such a fully sharp
hypothesis, all potential outcomes in  $\bm W$ can be imputed using
the observed outcome $\bm Y$. However, we are often interested in
testing more granular hypotheses that can only impute
part of $\bm W$. Such hypotheses are called "partially sharp" by \citet{zhangzhao2022}.  Conditioning is a useful tool to
exclude non-imputable potential outcomes. Specifically, a conditional
randomization test is defined as follows.

\begin{definition}\label{def:CRT}
  A \emph{conditional randomization test} (CRT) for the treatment  $\bm Z$ is defined by
  a finite partition $\mathcal{R} =
  \{\mathcal{S}_m\}_{m=1}^{M}$ of $\mathcal{Z}$ and
  a \emph{test statistics} $T = t(\bm Z, \bm W)$ where
  $t:\mathcal{Z} \times \mathcal{W} \to \mathbb{R}$ is a
  (measurable) function of $\bm Z$  and $\bm W$. With an abuse of notation, let
  $\mathcal{S}_{\bm z}$ be the set in the partition $\mathcal{R}$
  that contains $\bm z$. The \emph{p-value} of the CRT is then
  given by
  \begin{equation}
    \label{eq:p-value}
    P = p(\bm Z, \bm W) = \mathbb{P} \{t(\bm Z^{*}, \bm W)
    \leq t(\bm Z, \bm W) \mid \bm Z^{*}\in \mathcal{S}_{\bm
      Z}, \bm Z, \bm W \},
  \end{equation}
  where $\bm Z^{*}$ is an independent copy of $\bm Z$ conditional on
  $\bm W$, that is, $\bm Z^{*}$ has the same
  distribution as $\bm Z$ but is independent of $\bm Z$ given $\bm W$. 
\end{definition}

In \eqref{eq:p-value}, the CRT compares the test statistic evaluated at the observed
treatment $\bm Z$ with all other potential treatment assignments $\bm
Z^{*}$ that are in the same equivalence class as $\bm Z$. The
equivalence class can be represented by a function $v: \mathcal{Z} \to
[M]$ such that $\mathcal{S}_m$ is the preimage of $\{m\}$, i.e.,
\begin{equation}\label{equ:s_m_v}
\mathcal{S}_m = v^{-1}(\{m\}) = \{\bm z \in \mathcal{Z}: v(\bm z) =
m\},\  \text{for all $m\in [M].$ }
\end{equation}
With $T^* = t(\bm Z^*, \bm W), V^*= v(\bm Z^*)$ and $V= v(\bm Z)$, we can rewrite the p-value in \eqref{eq:p-value} as
\begin{equation}\label{equ:p_m_v}
P = \mathbb{P} (T^*
    \leq T \mid V^{*} = V, \bm Z, \bm W ).
\end{equation}


Note that our p-value is a function of the potential outcomes
schedule $\bm W$ instead of the observed outcomes $\bm Y$.
This representation guarantees the CRT to be valid given $V$, as stated in the theorem below. With this in mind, the role of the null hypothesis is to impute the missing potential outcomes in $\bm W$ and compute the p-value, which is not always possible when the causal hypothesis is partially sharp. In this case, the conditioning statistic $V$ needs to be
carefully chosen. Examples of $V$ we will below include the (random) set of units that are treated at a given time point and the subvector $\bm Z_{\mathcal{I}}$ for some pre-selected subset $\mathcal{I}$.

\begin{theorem}\label{thm:valid}
  Under Assumption \ref{assump:randomization},  the p-value in
  \eqref{eq:p-value} satisfies, for any $\alpha \in [0,1],$
  \begin{equation}
    \label{eq:stoc-domi-sigma}
    \P\left\{ P(\bm Z, \bm W) \leq \alpha \mid V, \bm
      W\right\} := \sum_{m=1}^{M} 1_{\{\bm Z \in \mathcal{S}_m\}}
    \P\left\{ P(\bm Z, \bm W) \leq \alpha \mid
      \bm Z \in \mathcal{S}_m, \bm W\right\}
    \leq \alpha.
  \end{equation}
  In consequence, the CRT is valid in the sense that  $\mathbb{P}\left\{
    P(\bm Z, \bm W)\leq \alpha\right\}\leq \alpha$ for all
  $\alpha\in[0,1]$.
\end{theorem}

\Cref{thm:valid} is a well-known result in randomization inference, and a proof is included in
\Cref{sec:proof-crefl-2} for completeness.
By randomizing the rejection, the inequality
\eqref{eq:stoc-domi-sigma} can be replaced by equality; see \citet[eq.\
5.51]{lehmann2006testing}. For this reason, we refer to the
CRT as ``nearly exact'' and the multiple CRTs introduced below as
``nearly independent''.

\section{Sequential conditional randomization tests}
\label{sec:sequential-crts}

Next, we consider the problem of constructing $K$ CRTs for $K$
potentially different null hypotheses. Following \eqref{equ:s_m_v} and
\eqref{equ:p_m_v}, we define the CRTs using $K$ conditioning statistics:
\[
V^{(k)} = v^{(k)}(\bm Z) \  \text{ for some function $v^{(k)}$ on } \mathcal{Z},~k \in [K].
\]
The sequential CRTs are constructed by conditioning on more and more information as $k$ increases.
Specifically, we compute the p-value in the $k$th CRT using the distribution of
$V^{(k)}$ given $\bm V^{[k-1]}$ and a test statistics $t^{(k)}(\bm V^{[k]},\bm W)$:
\begin{equation}
  \label{eq:sequential-crt}
  P^{(k)} = p^{(k)}(\bm V^{[k]}, \bm W) = \mathbb{P} \left\{t^{(k)}(\bm V^{*[k]}, \bm W)
  \leq t^{(k)}(\bm V^{[k]}, \bm W) \mid \bm V^{*[k-1]} = \bm V^{[k-1]}, \bm Z, \bm W \right\},
\end{equation}
where $V^{*(k)} = v^{(k)}(\bm Z^*)$ and $\bm Z^*$ is an i.i.d. copy of
$\bm Z$.

The next result shows that this construction generates nearly
independent p-values regardless of how the conditioning statistics
$V^{(1)},\dots,V^{(K)}$ are chosen. Note that sometimes it is
desirable to condition additionally on a global statistic $V^{(0)}$
in all the tests. Then the test statistics $t^{(1)},\dots,t^{(K)}$ may
further depend on $V^{(0)} = v^{(0)}(\bm Z)$ and
equation \eqref{equ:dominance-3} below holds given $V^{(0)}$.




\begin{theorem} \label{prop:nested}
Under Assumption \ref{assump:randomization} and the setup above, the p-values $P^{(1)}, \dotsc,
  P^{(K)}$ defined in \eqref{eq:sequential-crt} are valid and
  nearly independent in the sense that
\begin{equation}\label{equ:dominance-3}
  \mathbb{P}\Big( P^{(1)} \leq \alpha^{(1)},\dotsc, P^{(K)} \leq \alpha^{(K)} \mid \bm W  \Big)
    \leq
  \prod_{k=1}^{K}\alpha^{(k)}~\text{for all $\alpha^{(1)},\dots,\alpha^{(K)} \in [0,1]$.}
\end{equation}
\end{theorem}

\begin{proof}
 Suppose $K =
2$. Let $\psi^{(k)}(\bm V^{[k]}, \bm
W) = 1_{\{P^{(k)} \leq \alpha^{(k)}\}}$ be the test function.
By using the law of iterated
expectations and \Cref{thm:valid}, 
\begin{align*}
  &\  \mathbb{P}\left\{P^{(1)} \leq \alpha^{(1)}, P^{(2)} \leq \alpha^{(2)} \ \Big|\ \bm W\right\} \\
  =&\ \mathbb{E}\left[ \psi^{(1)}(\bm V^{(1)}, \bm W) \mathbb{E}\left\{ \psi^{(2)}(\bm V^{[2]}, \bm W)  \ \Big| \
       V^{(1)} \right\} \ \Big| \ \bm W\right] \\
  \leq&\  \alpha^{(2)} \mathbb{E}\left\{\psi^{(1)}(V^{(1)}, \bm W) \ \Big| \
         \bm W \right\} \\
  \leq&\   \alpha^{(1)} \alpha^{(2)}.
\end{align*}
By using the above argument sequentially, this proof easily
generalizes to $K > 2$.
\end{proof}



In the next three sections, we will illustrate this
sequential construction of CRTs using the three examples introduced in
\Cref{new_intro}. Specifically, in \Cref{sec:evid-fact-observ} we will give a simple explanation of Rosenbaum's evidence factors for observational studies, without using any special structure of permutation groups as suggested by \citet{rosenbaum2017general}. In
\Cref{sec:swd,sec:spillover}, we will develop some new statistical
methods to test lagged and spillover treatment effects. In
\Cref{sec:multiple-crt}, we will give a more general theorem that does
not require a strictly sequential construction of CRTs.

\section{Example 1: Evidence factors for observational studies}
\label{sec:evid-fact-observ}


\citet{rosenbaum2010evidence,rosenbaum2017general,rosenbaum2021replication}
developed a concept called ``evidence factors'' for observational
studies, which are essentially pieces of evidence regarding
different aspects of an abstract causal hypothesis. Roughly speaking,
each evidence factor is characterized by a p-value upper bound
for a different hypothesis and the factors together satisfy \eqref{equ:dominance-3}, as we detail below.

Consider the example in
\citet{rosenbaum2017general} that concerns the causal effect of
smoking on periodontal disease. Each subject in this analysis
is categorized as a ``non-smoker'', ``light smoker'', or ``heavy
smoker''. The main idea of \citet{rosenbaum2017general} is to
represent exposure as two indicators, one for smoking and one for
heavy smoking. That is, let the exposure be $\bm Z = (\bm Z^{(1)},
\bm Z^{(2)}) $, where $\bm Z^{(1)}$ and $\bm Z^{(2)}$ are
$N$-dimensional binary variables, $Z^{(1)}_i = 1$ indicates subject $i$ is
a smoker, and $Z^{(2)}_i = 1$ indicates subject $i$ is a
heavy smoker. Rosenbaum then constructed a
randomization test by comparing the outcomes of the smokers with the
non-smokers, and another randomization test by comparing the outcomes
of the heavy smokers with the light smokers.

Rosenbaum showed that the two tests are nearly independent by expressing the permutation group of $(\bm Z_1, \dots, \bm Z_n)$ as a knit product of two smaller permutation groups, one for $(Z^{(1)}_1, \dots, Z^{(1)}_N)$ and one for $(Z^{(2)}_1, \dots, Z^{(2)}_N)$. A group $G$ is said to be the knit product of two subgroups $G_1$ and $G_2$, if every element $g \in G$ can be uniquely represented as $g = g_1 g_2$ where $g_1 \in G_1$ and $g_2 \in G_2$. Intuitively, any permutation of the exposure status of the subjects can be obtained by first permuting the non-smokers with the smokers, then permuting the light smokers with the high smokers (within the set of smokers). Thus, the two permutation tests should be independent. Because smoking is not truly randomized in an observational study, \citet{rosenbaum2017general} matched pairs of smokers and non-smokers in terms of age, gender, education, income, and ethnicity before running the randomization tests.

Rosenbaum's observation is intuitive, but the argument is tailored to
permutation tests. The mathematical proof in
\citet{rosenbaum2017general} is also quite involved. In view of our
theory in \Cref{sec:sequential-crts}, this can be simplified, as Rosenbaum's two evidence factors are exactly the
sequence of CRTs in \Cref{sec:sequential-crts} with $\bm V^{(1)} = \bm
Z^{(1)} $ and $\bm V^{(2)} = \bm Z^{(2)}$. By \Cref{prop:nested}, the
p-value upper bounds are nearly independent.
Suppose the first CRT assumes a collection $\Pi^{(1)}$ of
randomization distributions of $\bm Z^{(1)}$ and the second CRT
assumes a collection $\Pi^{(2)}$ of randomization distributions of
$\bm Z^{(2)}$ given $\bm Z^{(1)}$. Rosenbaum obtained closed-form
expressions of the p-value upper bounds
\[
  P^{(k)} = \sup_{\pi^{(k)} \in \Pi^{(k)}} p^{(k)}(\bm
  Z, \bm W; \pi^{(k)}),~k=1,2,
\]
where $p^{(k)}(\bm Z, \bm W; \pi^{(k)})$ is the p-value of the
$k$th CRT under the distribution $\pi^{(k)}$.
It directly follows from \Cref{prop:nested} that these p-values are valid and nearly independent for any distribution in $\pi = (\pi^{(1)}, \pi^{(2)}) \in \Pi^{(1)} \times \Pi^{(2)}$, where $\pi(\bm z) = \pi^{(1)}(\bm z^{(1)}) \cdot \pi^{(2)}(\bm z^{(2)} \mid \bm z^{(1)})$. That is,
\begin{equation}\label{equ:paul_argument}
\sup_{\pi \in \Pi^{(1)} \times \Pi^{(2)}} \mathbb{P}_{\pi}\left(  P^{(1)} \leq
  \alpha^{(1)},P^{(2)} \leq \alpha^{(2)}  \right)  \leq \alpha^{(1)} \alpha^{(2)}. 
\end{equation}
Thus, we can view $P^{(1)}$ and $P^{(2)}$ as independent pieces of
evidence about two related causal hypotheses on the effect of smoking
on periodontal disease. This argument does not rely on any specific
algebraic structure of permutation groups (such as the knit product
used by Rosenbaum) other than that they are constructed sequentially
as described in \Cref{sec:sequential-crts}.

\section{Example 2: Testing lagged treatment effects in stepped-wedge
  trials} \label{sec:swd}




\subsection{Sequential tests for lagged treatment effects}
\label{sec:sw-nested-crts}

As introduced in \Cref{label:s_exp}, the stepped-wedge design
randomizes cross-over times and thus presents an opportunity to
investigate the time lag of treatment effect. In this section, we will
describe a sequential construction of CRTs that will allow us to pool
information from different time points. For simplicity, we assume that
the randomization happens at the unit level, and note here that the method
below is applicable to cluster-randomized trials by using
aggregated cluster-level outcomes
\citep{middleton2015unbiased,thompson2018robust}.

We will consider a stepped-wedge trial on $N$ units over an evenly
spaced time grid $[T] = \{1,\cdots,T\}$. At time $t$, a total of $N_t
\geq 0$ units cross over from control to treatment, and for simplicity
we assume $\sum_{t=1}^T N_t = N$, so all units are treated when the trial is finished. The cross-overs can be represented
by a binary $N \times T$ matrix $\bm Z$, where $Z_{it} =
1$ indicates that unit $i$ crosses over from control to treatment at
time $t$. Usually, the units stay in the treatment arm after crossing
over, but this framework also covers the case of ``one-shot
treatment'' by redefining treatment as whether the unit has received
the one-shot treatment already. We consider the simple scenario of
``complete randomization'', where $\bm Z$ is uniform assigned over
  $\mathcal{Z} = \left\{\bm z \in \{0,1\}^{N\times T}: \bm
    z\bm 1 = \bm 1, \bm 1^{\top} \bm z  = (N_1,\cdots,N_T)\right\}$
  ($\bm 1$ is a vector of ones).
Thus, the treatment assignment mechanism is given by $\pi(\bm z) = 1 /
|\mathcal{Z}|$ for $\bm z \in \mathcal{Z}$ where $|\mathcal{Z}| = N! /
(N_1! \dotsc N_T!)$.

After the cross-overs at time $t$,  the experimenter
immediately measures the outcomes of all units, denoted by
$\bm Y_t \in \mathbb{R}^N$. The corresponding vector of potential
outcomes under the assignment $\bm Z = \bm z$ is denoted by $\bm
Y_t(\bm z) = (Y_{1t}(\bm z), \dotsc, Y_{nt}(\bm z))$. Let $\bm Y =
(\bm Y_1,\dotsc, \bm Y_T)$ be all the realized outcomes and $\bm Y(\bm z)$ be the
collection of potential outcomes under treatment assignment $\bm z \in
\mathcal{Z}$. As before, we assume the observed outcomes are
consistent, i.e., $\bm Y = \bm Y(\bm Z)$. In this setting,
several authors have developed unconditional randomization tests for the null hypothesis of no treatment effect whatsoever
\citep{ji2017randomization, wang2017use, hughes2020rob}:
\begin{equation}
  \label{eq:swd-sharp-null}
  H: \bm Y(\bm z) = \bm Y(\bm z^{*}),\quad \forall \bm z, \bm z^{*} \in
  \mathcal{Z}, i \in [N], \text{ and } t\in [T].
\end{equation}
Below, we will focus on less restrictive hypotheses about lagged
treatment effects.

To make the problem more tractable, we assume no interference between
units and no anticipation effect from treatment assignment in the
future. See \citet{athey2018design} for some discussion on this
assumption.

\begin{assumption}[No interference or anticipation]\label{assumption:noanti}
  For all $i\in [N]$ and $t\in[T]$, $Y_{it}(\bm z)$ only depends on $\bm
  z$ through $\bm z_{i,[t]}$, i.e.\ the treatment history for unit $i$
  up to time $t$.
\end{assumption}

Given this assumption, we can rewrite $Y_{it}(\bm z)$ as
$Y_{it}(\bm z_{i,[t]})$. Given a fixed time lag $0 \leq l \leq
T - 1$, we have a sequence of hypotheses about lag-$l$ treatment
effect:
for $t = 1,\dotsc,T-l$,
\begin{equation}\label{equ:null_st}
  H^{(t)}: Y_{i,t+l}(\bm 0_{t-1},1,\bm 0_{l})- Y_{i,t+l}(\bm 0_{t+l}) = \tau_{l},\ \forall i\in [N].
\end{equation}
In words, $H^{(t)}$ means that crossing over from control to treatment at time $t$ has a constant
treatment effect $\tau_l$ on the outcome at time $t+l$.

Next, we will construct a sequence of CRTs for
$H^{(1)},\dotsc,H^{(T-l)}$ and combine them
to test the intersection hypothesis $\cap_t H^{(t)}$. This
intersection hypothesis essentially asserts that
the $l$-lagged treatment effect is always $\tau_l$, which is still
much weaker claim than the fully sharp hypothesis $H$ in
\eqref{eq:swd-sharp-null}.
Note that it is difficult to test the intersection $\cap_t H^{(t)}$ via a single CRT using outcomes across all time points.
For every unit $i$, there exists some time steps $t$ such that both outcomes in $H^{(t)}$ are unobserved. Therefore, the CRT using all the potential outcomes in $\cap_t H^{(t)}$ is
non-computable.
This motivates us to test each $H^{(t)}$ individually using outcomes at fewer time steps.


We first show that the most straightforward CRTs
for $H^{(1)},\dots,H^{(T-l)}$  are not independent.
The hypothesis $H^{(t)}$ in \eqref{equ:null_st} only
concerns the potential outcomes of $Y_{i,t+l}$ when the cross-over occurs
at time $t$ or after $t+l$. Thus, in order to test $H^{(t)}$, we can
only compare units that crossed over at time $t$ with those crossed
over after time $t+l$. We denote this random set of units as
\begin{equation}
  \label{eq:swd-naive-conditioning}
v^{(t)}(\bm Z) = \left\{i \in [N]: \bm Z_{i,[t+l]} = (\bm
    0_{t-1},1,\bm  0_{l}) \text{ or } \bm 0_{t+l} \right\}.
\end{equation}
Because the treatment assignment mechanism $\pi$ is uniform over
$\mathcal{Z}$, given the random set of units $v^{(t)}(\bm Z)$, the
treatment assignment $\bm Z$ is still uniformly distributed over all
permutations of the vectors $\bm Z_{i, [t+l]}$ for $i \in v^{(t)}(\bm
Z)$. This shows that under Assumptions \ref{assump:randomization} and
\ref{assumption:noanti}, the CRT given $ v^{(t)}(\bm Z)$ is a
two-sample permutation test between the ``treated'' group that crosses
over at time $t$ and the ``control'' group that crosses over after
time $t+l$.

\begin{figure}[t]
  \centering
  \begin{subfigure}[b]{\textwidth}  \centering
    \includegraphics[width = 0.8\textwidth]{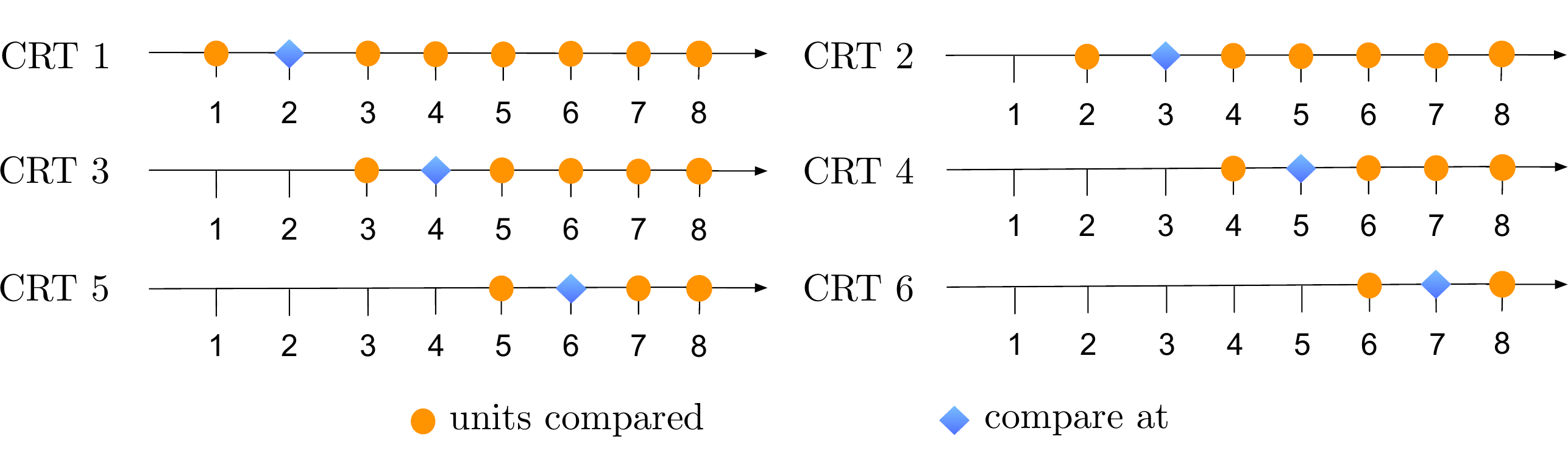}
    \caption{Non-nested CRTs.} \label{fig:swd-crt-1}
  \end{subfigure}
  \begin{subfigure}[b]{\textwidth}\centering
    \includegraphics[width = 0.8\textwidth]{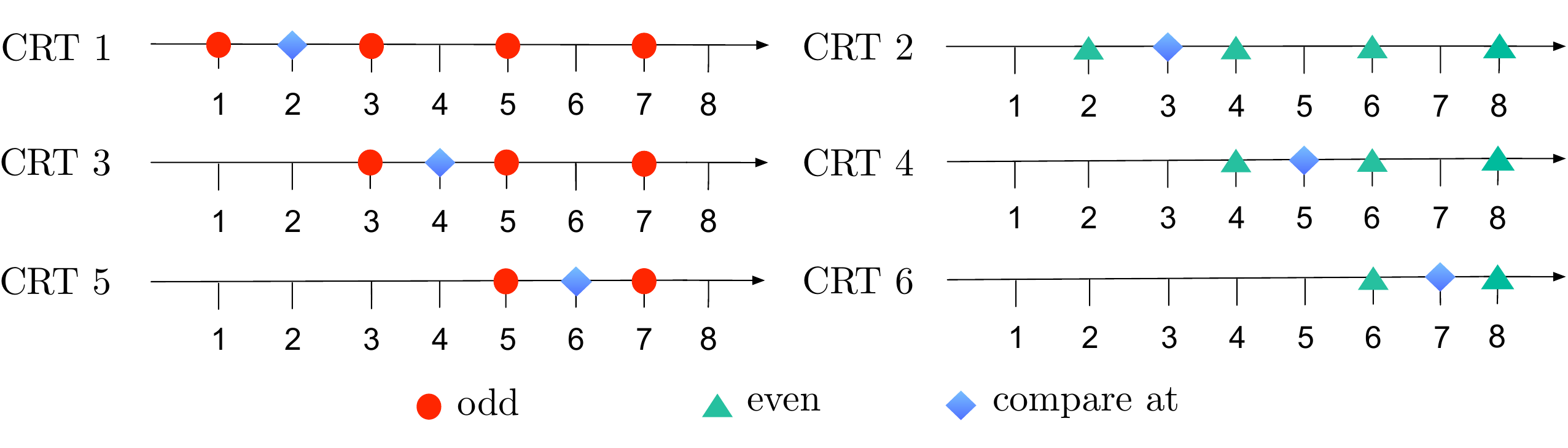}
    \caption{Nested CRTs.} \label{fig:swd-crt-2}
  \end{subfigure}
  \caption{Illustration of CRTs for lagged effect (lag
    $l=1$)
    in a stepped-wedge trial.}
  \label{fig:swd-crt}
\end{figure}

However, the above construction does not yield nearly independent CRTs
for $l \geq 1$. We visualize this issue with $l=1$ in
\Cref{fig:swd-crt-1}. In this case, the CRT for
$H^{(t)}$ is a permutation test that compares the time $t+1$ outcome
for  crossing over at $t$ with those crossing over after
$t+1$. Specifically, the first CRT compares cross-overs at time
$1$ with cross-overs at time $3,4,\dotsc$, the second CRT compares
cross-overs at time $2$ and cross-overs at time
$4,5\dotsc$, and so on. Thus, the conditioning statistics are not
nested in the sense that $\bm V^{(1)} = v^{(1)}(\bm Z)$ can not be
written as a function of $\bm V^{(2)} = v^{(2)}(\bm Z)$, and vice
versa. In other words, suppose we let the second CRT condition on
$
v^{(1)}(\bm Z)$ and $
v^{(2)}(\bm Z)$,
as prescribed by \Cref{prop:nested}.
The second CRT will have a degenerated and powerless randomization distribution
because the intersection  $v^{(1)}(\bm Z)\cap v^{(2)}(\bm Z)$
exactly identifies the units that crossover at time $2$.



To resolve this problem, we can gain some intuition by considering the
$l=0$ case, where the above construction does generate independent
tests. Observe that the stepped-wedge trial can be viewed as a
sequentially randomized experiment: the assignments $\bm Z$ can be
thought of as randomly starting the treatment for $N_t$ untreated
units at every $t = 1,\dotsc,T$. More precisely,  let $\bm Z_t$ denote
the $t$-th column of $\bm Z$,
$\bm z_{[t]} = (\bm z_1, \dotsc, \bm z_t) \in \{0,1\}^{N \times t}$ and $N_{[t]} = \sum_{s=1}^t
N_s$. We can decompose $\pi$ as
\begin{equation}\label{equ:swd_sequence}
  \frac{N_1! \dotsb N_T!}{N!} =\pi(\bm z) = \pi_1(\bm z_1) \pi_2(\bm z_2 \mid \bm z_1) \dotsb
  \pi_T(\bm z_T \mid \bm z_{[T-1]}),
\end{equation}
where $\pi_t$ is a complete randomization that assigns $N_t$
out of $N - N_{[t-1]}$ units to treatment:
\begin{equation}\label{equ:swd_sequence_2}
\pi_t(\bm z_t \mid \bm z_{[t-1]}) = \frac{N_t! (N - N_{[t]})!}{(N -
    N_{[t-1]})!}, ~\forall \\ \bm{z}_t \in \{0,1\}^{N}
  \text{ such that } (\bm 1_N - \bm z_{[t-1]} \bm 1_{t-1})^T \bm z_t = N_t,
\end{equation}
i.e.\ $\bm Z_t$ is uniformly distributed over all assignments with
$N_t$ crossovers. In the case of $l=0$, the conditioning statistics $
\bm V^{(t)} = \bm Z_t$ are nested: the first CRT compares cross-overs at time $1$ with
cross-overs at time $2,3,\dotsc$, the second CRT compares cross-overs
at time $2$ with cross-overs at time $3,4\dotsc$, and so on. By
\Cref{prop:nested}, these CRTs are nearly independent. 

The above observation inspires the following construction of nearly
independent CRTs for lag $l \geq 1$. Let us first consider $l=1$. In this
case, we divide the $T$ time steps into two groups:
\begin{equation}\label{equ:c_odd_even}
\mathcal{C}_{\text{odd}} =\{1,3,\dotsc\}\ \text{ and } \ \mathcal{C}_{\text{even}} =
\{2,4,\dotsc\}.
\end{equation}
Denote the set of units that cross over at odd time points as
\[  V^{(0)} =  v^{(0)}(\bm Z)= \bigg\{i \in [N]: \sum_{t\in
  \mathcal{C}_{\text{odd}} } Z_{i,t} = 1  \bigg\}.
\]
Then we divide the treatment variables
$\bm Z$ into two sets over time:
\begin{align*}
&  \bm V_{\text{odd}}  =  (\bm V^{(1)},\bm V^{(3)} ,\dotsc)  =  (\bm Z_1, \bm Z_3,\dotsc) \  \text{ and }\ \bm V_{\text{even}}  =  (\bm V^{(2)}, \bm V^{(4)},\dotsc)  =  (\bm Z_2, \bm Z_4,\dotsc).
\end{align*}
We have $\bm V_{\text{odd}} \independent \bm V_{\text{even}} \mid V^{(0)}$ given the following factorization of the randomization distribution:
\begin{equation}\label{equ:ci_pi}
\begin{split}
\pi(\bm z) & =
\pi(v^{(0)})\pi_{\text{odd}}(\bm v_{\text{odd}} \mid v^{(0)})\pi_{\text{test}}(\bm v_{\text{even}}\mid v^{(0)})  \\
& \equiv \frac{N_{\text{odd}}! N_{\text{even}}!}{N!} \times \frac{\prod_{t\in \mathcal{C}_{\text{odd}} }N_t ! }{ N_{\text{odd}}!} \times  \frac{\prod_{t\in \mathcal{C}_{\text{even}} }N_t ! }{ N_{\text{even}}!},
\end{split}
\end{equation}
where $N_{\text{odd}} = \sum_{t\in \mathcal{C}_{\text{odd}} }N_{t}$ and $N_{\text{even}} = \sum_{t\in \mathcal{C}_{\text{even}} }N_{t}$. Observe that the conditional distributions $\pi_{\text{odd}}(\bm v_{\text{odd}} \mid v^{(0)})$ and $\pi_{\text{test}}(\bm v_{\text{even}}\mid v^{(0)})$ can be further factorized in the same way as \eqref{equ:swd_sequence} and \eqref{equ:swd_sequence_2}. This motivates us to construct two sequences of CRTs, one for the odd time steps and one for the even time steps.
Specifically, in the odd sequence, CRT 1 compares cross-overs at time $1$ with cross-overs at time $3,5,\dotsc$; CRT 3 compares cross-overs at
time $3$ and cross-overs at time $5,7,\dotsc$; and so on
(see \Cref{fig:swd-crt-2}).

The above sequence of tests follows the sequential structure in
\Cref{sec:sequential-crts}, where the $k$th CRT conditions on
$(V^{(0)}, \bm V^{(1)},\bm V^{(3)},\dotsc, \bm V^{(k-2)})$ for odd
$k.$ The even sequence can be constructed similarly, and the
$k$th CRT conditions on $(V^{(0)}, \bm V^{(2)},\bm V^{(4)},\dotsc, \bm
V^{(k-2)})$ for even $k.$
By $\bm V_{\text{odd}}
\independent \bm V_{\text{even}} \mid V^{(0)}$ and \Cref{prop:nested}, the p-values of all these CRTs are nearly independent given $V^{(0)}$. For example, when $T = 5$, our theorem implies that
\begin{align*}
  &\ \mathbb{P}( P^{(1)} \leq \alpha^{(1)}, P^{(2)} \leq \alpha^{(2)},
    P^{(3)} \leq \alpha^{(3)}, P^{(4)} \leq \alpha^{(4)} \mid V^{(0)} ) \\
  = &\ \mathbb{P}( P^{(1)} \leq \alpha^{(1)}, P^{(3)} \leq \alpha^{(3)}
      \mid V^{(0)} ) ~ \mathbb{P}( P^{(2)} \leq \alpha^{(2)}, P^{(4)}
      \leq \alpha^{(4)} \mid V^{(0)} )  \\
  \leq &\ \alpha^{(1)} \alpha^{(3)} \alpha^{(2)} \alpha^{(4)}. \tag*{(By
         \Cref{prop:nested})}
\end{align*}
These p-values can then be combined using Fisher's or Stouffer's
method to test the intersection hypothesis $\cap_t H^{(t)}$ about lag
$l = 1$ effect. They use all the observed outcomes over time, so in this sense, this modified sequential construction uses all the information in the data.


The above construction can be easily generalized to $l \geq 1$.
To test the
lag $l$ effect, we can emulate \eqref{equ:c_odd_even} to increase the gap and divide
$[T]$ into disjoint subsets:
\[
\mathcal{C}_1 =\{1,l+2,2l+3,\dotsc\},\mathcal{C}_2 =
\{2,l+3,2l+4,\dotsc\},\dotsc,\mathcal{C}_{l+1} = \{l+1, 2l + 2, \dots\},
\]
and let $\bm V^{(0)}$ be the random sets $\{i \in [N]: \sum_{t\in \mathcal{C}_{1} } Z_{i,t} = 1 \},\{i \in [N]: \sum_{t\in \mathcal{C}_{2} } Z_{i,t} = 1 \},\dotsc,$ which collects the crossover units in $\mathcal{C}_1,\mathcal{C}_2,\dotsc,$ respectively.
Pseudocode for this algorithm with a general time lag $l$ is given in \Cref{appendix:randomization}, and we will call it multiple conditional randomization tests (MCRTs) in the examples below.


\subsection{Real data examples}

We next demonstrate our method using two real stepped-wedge randomized
trials. Trial I was conducted to examine if a chronic care model was
more effective than usual care in the Netherlands between 2010 and 2012
\citep{muntinga2012implementing}. The study consists of 1,147 frail
older adults in 35 primary care practices.
The primary outcome in the dataset was quality of life measured by a
mental health component score in a 12-item Short Form
questionnaire (SF-12). Each adult's outcome was measured at a
baseline time point and every 6 months over the study. The trial
randomly assigned some practices to start the chronic care model every
6 months. Each practice is an experimental unit, i.e., $N = 35$ and $T = 5$.

Trial II was a stepped-wedge trial performed over the pain clinics of 17
hospitals to estimate the effectiveness of an intervention in reducing
chronic pain for patients \citep[Chapter 6.4]{twiskanalysis2021}. The
outcome was a pain score ranging from 1 (least pain) and 5 (most pain), and six
measurements were taken over time: one at the baseline and five more
equally spaced over a period of twenty weeks. After each measurement,
intervention was started in a few randomly selected
untreated hospitals. We considered each hospital
as an experimental unit and estimated the lag $0,1,2$ and $3$ effects
using hospital-level average outcomes ($N = 17, T = 6$).

In both trials, we implement all our CRTs as two-sample permutation
tests using the difference-in-means statistic. We combine the p-values
from MCRTs using Stouffer's method based on z-scores; this will be
denoted as MCRTs+Z.
We compare it with mixed-effects
models, with and without time-fixed effects.
The 90\%-confidence intervals (CIs) for the lagged treatment effects are reported in \Cref{fig:real_data}.
For Trial I in \Cref{fig:real_data}(a), the mixed-effects model without time-fixed effects suggests an improvement in the quality of life with the new chronic care model.  However, as noted by
\citet[Section
6.3.1]{twiskanalysis2021},
this model overlooks the gradual increase in the quality of life over time.
In contrast, both the mixed-effect model with time-fixed
effects and our method MCRTs+Z show that the treatment effect is not
statistically significant, and MCRTs+Z provides slightly wider
confidence intervals. The difference between mixed-effect models with
and without time fixed effects is even more pronounced in Trial II. In
\Cref{fig:real_data}(b), the confidence intervals generated by these
models for the same time lag do not overlap. MCRTs+Z generated
confidence intervals that lie in between, suggesting that the
effectiveness of the intervention may increase over time.
In conclusion, our method (MCRTs) provides a robust and nonparametric
alternative to existing mixed-effect models for inferring lagged
treatment effects. In \Cref{sect:sim_2_mixed}, we will investigate the coverage rates of the confidence intervals obtained by these methods using numerical simulations in cases where the mixed-effect models are misspecified.


\begin{figure}[t]
  \centering
  \includegraphics[width=0.60\linewidth]{./Figures/exp_real}
  \caption{Effect estimates from MCRTs+Z and mixed-effects models with and without time effect parameters: 90\%-confidence intervals (CIs) of lagged effects on real data collected from four different stepped-wedge randomized trials.}
  \label{fig:real_data}
  \centering
\end{figure}

\section{Example 3: Range of spillover effect} \label{sec:spillover}

\subsection{Sequential CRTs for spillover effects}\label{sect:scrt_spill}

We now turn to the problem of testing spillover effects described
in \Cref{label:c_exp}. To be more concrete, we consider the
setting with $N$ units in a spatial domain.
Let $\bm z = (z_1,\dots,z_N)$ denote the treatment assignment of all the units and
$Y_i(\bm z)$ denote the potential outcome of unit $i$ under treatment $\bm
z$. The hypothesis of no interference can be written as
\[
  H^{(1)}: Y_j(\bm z) = Y_j(\bm z') ~\text{if}~z_j = z_j'.
\]
\citet{athey2018exact} proposes to test $H^{(1)}$ on the potential
outcomes of a (randomly) selected subset of  ``focal units''
$\mathcal{J}$, while randomizing the treatment assignment of the other
units. By conditioning on the assignment of $\mathcal{J}$, all the
missing potential outcomes of $\mathcal{J}$ are imputable under $H^{(1)}$.
 Other related methods can be found in \citet{basse2019randomization}
 and \citet{puelz2019graph}.

In applications with limited resources for treatment allocation,
it is of interest to know not just whether interference exists but
also the approximate range of the spillover effect. We next propose a
sequential construction of CRTs that test the spillover effect at
increasing distances.

To formulate our hypotheses, suppose we are given a $N \times N$
symmetric matrix $\bm D$ of pairwise distances between the units and a sequence
of distance thresholds $0=\epsilon^{(0)}<\epsilon^{(1)} < \epsilon^{(2)}
< \dots < \epsilon^{(K)}$. Let $\bm B^{(-1)}$ be the $N\times N$ matrix
of zeros, $\bm B^{(0)}$ be the $N\times N$ identity matrix, and $\bm
B^{(k)}$ be the interference matrix with entries given by the
indicator of $\{0 \leq  D_{ij}\leq \epsilon^{(k)} \}$ for
$i,j \in [N]$ and $k \in [K]$. With this definition, the diagonal of $\bm B^{(k)}$ is
one for $k=0,1,\dots, K.$

We would like to test the following null hypothesis:
\begin{equation}\label{equ:h_k_spillover}
     H^{(k)}: Y_j(\bm z) = Y_j(\bm z')
     ~\text{if}~B^{(k-1)}_{ij} z_i = B^{(k-1)}_{ij}
     z'_i~\text{for all}~i \in [N],
\end{equation}
which means that there is no spillover effect between any pair of
units with distance larger than $\epsilon^{(k-1)}$ (using the
convention $\epsilon^{(-1)} = -\infty$).
Thus, $H^{(0)}$ corresponds to the hypothesis of no treatment effect whatsoever and  $H^{(1)}$ is the hypothesis of no interference mentioned above. In the literature, authors also consider the hypothesis of no spillover effect on the control, which adds a condition $z_j = 0$ to the hypothesis  $ H^{(k)}$ in \eqref{equ:h_k_spillover}. This modification does not make a technical difference to the CRTs (including ours). A valid CRT requires the selection of focal units to be independent of the observed assignment
$\bm Z$. To test this new hypothesis, we can only remove the treated units from the test statistics, given that the treatment assignment of the focal units $j$ is fixed in the CRTs.
To simplify the exposition, we will test the original
$ H^{(k)}$ in what follows.

To test different hypotheses $H^{(k)}$, we propose to randomize the
treatment assignment of different subsets of units $\mathcal
I^{(k)}$. Denote the subset
\[
  \mathcal{I}^{[k_1:k_2]} = \bigcup_{l = k_1}^{k_2}
  \mathcal{I}^{(l)}~\text{and}~\bm Z^{[k_1:k_2]} = \bm
  Z_{\mathcal{I}^{[k_1:k_2]}},~0 \leq k_1 \leq k_2 \leq K.
\]
Our first method can be viewed as a sequential extension of the method
in \citet{athey2018exact}. Let $\mathcal{J}$ be a chosen set of
``focal units''. Given $\mathcal{J}$, as $k$ increases, we iteratively define
\begin{equation}
  \label{eq:Ik}
  \mathcal{I}^{{(k)}} = \bigg\{i \in [N] \setminus \mathcal{I}^{[0:(k-1)]}:
  \sum_{j \in \mathcal{J}} B^{(k)}_{ij} \geq 1
  \bigg\},~0 \leq k \leq K.
\end{equation}
In words, as $k$ increases, we let $\mathcal{I}^{(k)}$ collect all the units that fall in
the $\epsilon^{(k)}$-balls centered at the focal units $\mathcal J$. But we have to remove $ \mathcal{I}^{[0:(k-1)]}$ because $H^{(k)}$ is about the spillover effect at a distance larger than $\epsilon^{(k)}.$ \Cref{fig:illustrate} illustrates
this unit selection procedure based on a static set of focal units.
The construction in \eqref{eq:Ik} implies that
\begin{equation}
  \label{eq:select-I-J}
  \sum_{i\in \mathcal{I}^{[k:K]}} B_{ij}^{(k-1)} =0~\text{for all}~j
  \in \mathcal{J},0 \leq k \leq K.
\end{equation}
Consequently, the potential outcomes of the focal units are
imputable under $H^{(k)}$ given $\bm z^{[0:(k-1)]}$:
\begin{equation}
  \label{eq:spillover-impute}
  Y_j(\bm z) = Y_j\quad\text{for all}~j \in
  \mathcal{J} \text{ and } \bm z~\text{such that}~\bm z^{[0:(k-1)]} = \bm
  Z^{[0:(k-1)]}.
\end{equation}

In principle, we can test $H^{(k)}$ using any test statistic
that only depends on the potential outcomes of $\mathcal{J}$. A
concrete example that will be used below is the
difference-in-means statistic
\begin{equation}\label{equ:statistics}
t^{(k)}  (\bm z, \bm W)    =  \frac{\sum_{j\in \mathcal{J}
                               }  A_j^{(k)}(\bm z^{(k)}) Y_j(\bm z)}{
                               \sum_{j \in \mathcal{J}   }
                               A_{j}^{(k)}(\bm z^{(k)})    }   -
                               \frac{\sum_{j\in \mathcal{J} }
                               \big\{1-  A_{j}^{(k)}(\bm z^{(k)})
                               \big\}Y_{j}(\bm z)  }{  \sum_{j \in
                               \mathcal{J}  }    \big\{1-
                               A_{j}^{ (k)}(\bm z^{(k)}) \big\}  },
\end{equation}
where $A_j^{(k)}(\bm z^{(k)}) $ is an indicator for $\sum_{i
  \in \mathcal{I}^{(k)}} z_i B_{ij}^{(k)} > 0$,
i.e., it indicates whether unit $j \in \mathcal{J}$ receives a
spillover effect from $\mathcal{I}^{(k)}$ if $H^{(k)}$ is not
true. When $k = 0$ so $\bm B^{(0)}$ is the identity matrix, this
reduces to $A_j^{(k)}(\bm z) = z_j$ for all $j \in \mathcal{J}^{(0)} =
\mathcal{I}^{(0)}$. Under $H^{(k)}$, equation
\eqref{eq:spillover-impute} shows that $T^{(k)}(\bm z, \bm W)$ only
depends on $\bm z$ through $\bm z^{[0:k]}$. Thus, by letting $\bm
V^{(k)} = \bm Z^{(k)}$ for $k = 0, 1,\dots, K$, we obtain a sequence of
nearly independent p-values as defined in \eqref{eq:sequential-crt}.

\begin{figure}[t]
  \centering
  \includegraphics[width=0.85\linewidth]{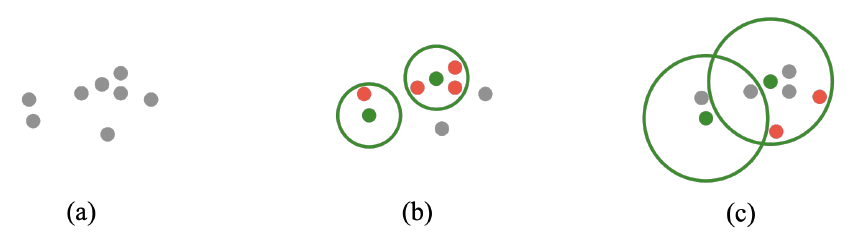}
  \caption{An illustration of the static unit selection procedure in \eqref{eq:Ik}. Panel (a) shows the units' locations. In panel (b), two green focal units $j\in \mathcal{J}$ has some red neighbours in its $\epsilon^{(1)}$-ball, which will be included in $\mathcal I^{(1)}$. In panel (c), we fix the treatment assignment of $\mathcal I^{(1)}$ (i.e. exclude them from $\mathcal I^{(2)}$) when testing the spillover effect at a larger distance $\epsilon^{(2)}$; the green focal unit on the left has no red neighbour as the distance increases.
      }
  \label{fig:illustrate}
  \centering
\end{figure}

\subsection{Adaptive selection of focal units}
\label{sec:adapt-select-focal}

A limitation of the above construction is that the set
$\mathcal{I}^{(k)}$ is often quite small for large distances
$\epsilon^{(k)}$; this will be exemplified using a real dataset in
\Cref{sec:real-data-example}. Heuristically, it is difficult to select
a single subset of focal units that lead to CRTs with decent power for
all the spillover distances. If the subset $\mathcal J$ is too dense (sparse)
in the spatial domain, the above tests have little power for spillover
effect at longer (shorter) distances.

Fortunately, the construction in \Cref{sect:scrt_spill} can be easily
modified to allow adaptive selection of focal units. The key idea is
to partition $[N]$ into subsets $\mathcal{I}^{(0)},\dots,\mathcal{I}^{(K)}$ first
and then use the largest possible set $\mathcal{J}^{(k)}$ of focal
units for every distance threshold $\epsilon^{(k)}$.
In the following, we first treat $\mathcal{I}^{(0)},\dots,\mathcal{I}^{(K)}$ as given and
discuss our choice of $\mathcal{J}^{(k)}$.

The CRT for $H^{(k)}$ uses the potential outcomes of the following set of ``focal units'':
\begin{equation}
  \label{eq:focal-units}
 \mathcal{J}^{(k)} = \left\{j \in [N]:  \sum_{i\in \mathcal{I}^{(k)}} B_{ij}^{(k)}\geq 1
\ \text{ and }  \sum_{i\in
     \mathcal{I}^{[k:K]}} B_{ij}^{(k-1)} =0
\right\}.
\end{equation}
The first condition in
\eqref{eq:focal-units} ensures that for every focal unit $j\in
\mathcal{J}^{(k)}$, there exists at least one unit $i\in
\mathcal{I}^{(k)}$ with distance $D_{ij}\in
(\epsilon^{(k-1)},\epsilon^{(k)}].$ In other words, the focal units are ``on the edge'' of experiencing a spillover
effect from $\mathcal{I}^{(k)}$ when we randomize $\bm Z^{[k:K]}$ and
$H^{(k)}$ is not true.


Like \eqref{eq:select-I-J},
the second condition in \eqref{eq:focal-units} implies that the distances between the focal units and $\mathcal{I}^{[k:K]}$
are larger than $\epsilon^{(k-1)}$.
The definition in  \eqref{eq:focal-units} implies that $\mathcal{J}^{(0)} = \mathcal{I}^{(0)}$ and $\mathcal{J}^{(k)} \subseteq
\mathcal{I}^{[0:(k-1)]} $ for $k \geq 1$.
Equation \eqref{eq:spillover-impute} holds with $\mathcal{J}$ replaced by $\mathcal{J}^{(k)}$.
Then we can use the same test statistic in \eqref{equ:statistics} but with $\mathcal{J}$ replaced by $\mathcal{J}^{(k)}$.
Then we have a sequence of nearly independent p-values
in \eqref{eq:sequential-crt} by choosing $\bm
V^{(k)} = \bm Z^{(k)}$ for $k = 0, 1,\dots,K$ as above.





Next, we describe a procedure for selecting the subsets
$\mathcal{I}^{(k)}$, with the aim of generating large
$\mathcal{I}^{(k)}$ and $\mathcal{J}^{(k)}$ to increase the
statistical power. We say $\mathcal{U}$ is an $\epsilon$-cover of a
set of units $\mathcal{M} \subseteq [N]$, if the $\epsilon$-balls of
the units in $\mathcal{U}$ cover all units in $\mathcal{M}$. Finding
the minimum cover is an NP-hard problem, but an approximate solution
can be obtained by solving its linear program relaxation.
Let $\mathcal{I}^{[1:0]}=\emptyset.$
As $k$ increases from 1 to $K,$ our algorithm iteratively finds the minimum $\epsilon^{(k)}$-cover $\mathcal{U}^{(k)}$ of the units $[N] \setminus \mathcal{I}^{[1:(k-1)]}$ and selects its subset
\begin{equation}\label{equ:i_k}
\mathcal{I}^{(k)} = \bigg\{i \in \mathcal{U}^{(k)} : \sum_{j \neq
  i}B_{ij}^{(k)}\geq  1 \bigg\},
\end{equation}
which are essentially the units whose $\epsilon^{(k)}$-balls contain not only themselves. Finally, we let $\mathcal{I}^{(0)} = [N]\setminus \mathcal I^{[1:K]}$.  Our algorithm ensures that  $\mathcal{I}^{(0)},\mathcal{I}^{(1)},\dots,\mathcal{I}^{(K)}$ are disjoint and form a partition of $[N]$, which is required in our construction of sequential CRTs.

Through the definition in \eqref{equ:i_k}, all the units in $\mathcal{I}^{[1:k-1]}$ have at least one neighbours left in
$[N] \setminus \mathcal{I}^{[1:(k-1)]}$.
Heuristically, solving the covering problem with and without  $\mathcal{I}^{[1:k-1]}$ will lead to a similar cover  $\mathcal{U}^{(k)}$, given that the $\epsilon^{(k)}$-balls get larger as $k$ increases. In other words, every $\mathcal{U}^{(k)}$ remains as an approximate cover of all the units in $[N].$ Consequentially, there are always focal units $\mathcal{J}^{(k)}$ to use in the neighbourhoods of
$\mathcal{I}^{(k)}\subseteq \mathcal{U}^{(k)}.$
Below, we compare this adaptive
approach of selecting focal units with the static approach in
\Cref{sect:scrt_spill} on a real-world dataset.

\subsection{A real data example}
\label{sec:real-data-example}

We illustrate our method with a real data example
\citep{jayachandran2017cash}. The dataset comprises 121 villages
located in Hoima district and northern Kibaale district in Uganda,
with 60 of these villages randomly assigned to receive annual payments
for forest preservation from 2011 to 2013. The
cash payments should reduce deforestation in the treated units, but one
concern is that they may increase deforestation in nearby,
untreated villages. Following \citet{wang2020design}, we
use change in the land area covered by trees as the outcome, which is
measured from satellite images \citep{hansen2013high}. We set $\epsilon^{(k)} = 3k$
and $K = 4$. The minimum set covering
problem is solved approximately using \texttt{setcover}
in the R-package \texttt{adagio} \citep{borchers2016adagio}.

\begin{figure}[t]
  \centering
    \includegraphics[width = 0.98\textwidth]{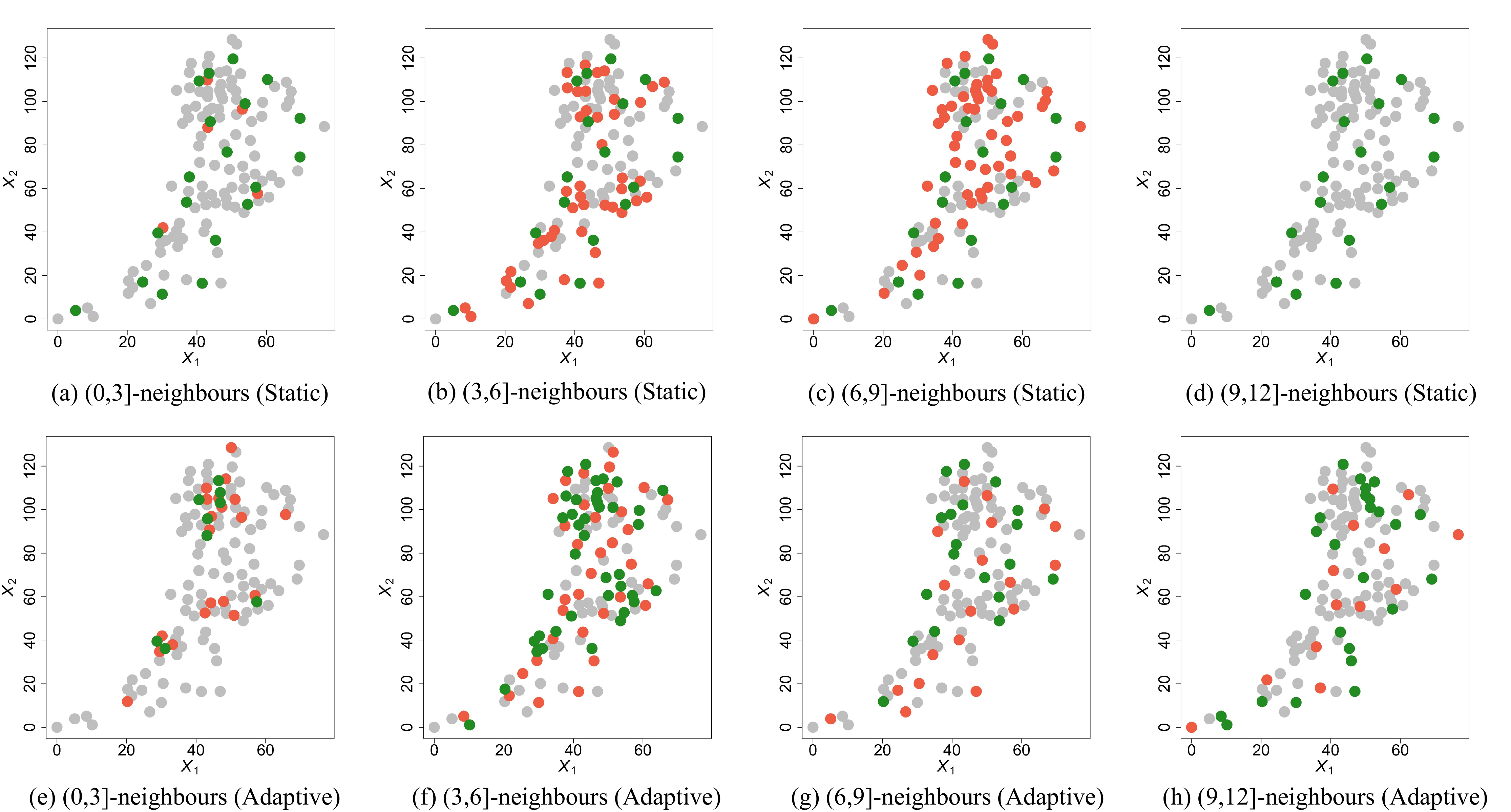}
    \caption{ The subsets of units $\mathcal I^{(k)}$ (red) and
      $\mathcal J^{(k)}$ (green) created by the baseline method
      (fixing the same subset of focal units in all tests) and our
      sequential covering procedure.
    }
  \label{fig:interference_0}
\end{figure}

To test the spillover effects within a distance of 12,
the static approach fixes the focal units in all the CRTs to the minimum 9-cover $\mathcal{U}$ of the units in $[N]$, which are those green points in each upper panel of \Cref{fig:interference_0}. We then select $\mathcal I^{(1)}$ to be the 3-neighbours of $\mathcal{U}$, which are the red dots in panel (a).
Recall that $H^{(k)}$ is about no spillover effect at a distance larger than $3(k-1).$ In its CRT, we take $\mathcal  I^{(k)}$ as the $3k$-neighbours of $\mathcal{U}$ but with $\mathcal I^{[k-1]}$ removed.  As shown in panels (a) and (d), this method leads to a small $\mathcal I^{(1)}$ and empty $\mathcal I^{(4)}$, respectively.  In comparison, the lower panels of \Cref{fig:interference_0} show that our adaptive method generates large enough $\mathcal{I}^{(k)}$ and  $\mathcal{J}^{(k)}$ while keeping the CRTs nearly independent. The distances between red and green units increase in the panels from left to right, indicating that we are testing the spillover effect at greater distances.

\begin{table}[t]
\caption{P-values  for testing spillover
  effects at different distances in the forest example.}
\label{tab:p_spillover}
\centering
\resizebox{0.56\columnwidth}{!}{%
\begin{tabular}{ |c|c|c|c|c | }
\hline
 \diagbox{Method}{Distance}  & 3 & 6 & 9  & 12
                                               \\
\hline
Static  & 0.564  &  0.056  & 0.404 & 1.000 \\
 MCRTs    & 0.013  & 0.120  & 0.503  & 0.713 \\
 MCRTs+F  & 0.058  & 0.392  & 0.726  & 0.713 \\
\hline
\end{tabular}
}
\end{table}

\Cref{tab:p_spillover} shows the p-values produced by the static method, our method
MCRTs and its variant MCRTs+F. MCRTs+F tests $H^{(k)}$ by combining the p-values $P^{(k)},\dotsc, P^{(4)}$ from MCRTs using Fisher's method. This strategy assumes that the spillover effect is decreasing as the spillover distance increases.
MCRTs and MCRTs+F agree that the payments can encourage more forest conservation within a distance of 3, and there is no displacement of deforestation. The combined p-values in MCRTs+F are larger than the ones in MCRTs. This is because the statistical power for testing spillover effects is complicated. It depends not only on the spillover distance but also on other factors, such as units' location and the design.
Nevertheless, the simulations in \Cref{sect:sim_interference} show that combining p-values can be useful when the CRT is not well powered at some distance.

\section{Numerical experiments}
\label{sec:experiment}

\subsection{Stepped-wedge design: size and power}\label{sim:swd}

We examine the performance of the MCRTs method in \Cref{sec:swd} by
comparing it with Bonferroni's Method and checking its validity in the
presence of unit-by-time interactions. We use the notations MCRTs+F
and MCRTs+Z to represent the combination of MCRTs with Fisher's
combination method \citep{fisher1925statistical} and Stouffer's method
using the Z-scores, respectively; more details can be found in \Cref{sec:combine}. In all the CRTs, we use the
difference-in-means statistic and 1000 permutations.

As shown in \Cref{fig:swd-crt}, the non-nested CRTs have larger control groups compared to our nested CRTs created in MCRTs. However, the non-nested CRTs are not independent, and there are limited ways to combine them. The most commonly used method in this case is Bonferroni's method, which rejects the intersection of $K$ hypotheses if the smallest p-value is less than $\alpha/K$.
A remaining question is whether the reduced sample size in MCRTs
can indeed be compensated by combining p-values.
To answer this question empirically, we conducted a simulation with variations in
the sample size $N$, trial length $T$, time lag $l$, and
effect sizes $\tau_l$, respectively.

\begin{figure}[t]
  \centering
  \includegraphics[width=0.90\linewidth]{./Figures/exp_power}
  \caption{Performance of MCRTs+F, MCRTs+Z and Bonferroni's method: type I error rates and powers in testing lagged effects at five different numbers of units, numbers of time steps, time lags and effect sizes. The results were averaged over 1000 independent runs.}
  \vspace{-8pt}
  \label{fig:error_power}
  \centering
\end{figure}

We fix $N_1=
\dotsc = N_T = N/T$ and $N_T=N-\sum_{t=1}^{T-1}N_t$ in the stepped-wedge design.
Let $A_i$ denote the treatment starting
time of unit $i$, i.e., the index of the non-zero entry of $Z_i$.
The outcomes were generated by a linear mixed-effects model,
\[
Y_{it} = \mu_i + 0.5(X_i+t)   + \sum_{l=0}^{T-1}1_{\{A_i-t =l
  \}}\tau_l + \epsilon_{it},~i=1,\dotsc,N,~t=0,\dotsc,T,
\]
where $\mu_i\sim \mathcal N(0,0.25)$, $X_i\sim \mathcal N(0,0.25)$ and
$\epsilon_{it}\sim \mathcal N(0,0.1)$. This assumes that the baseline outcome
$Y_{i0}$ is measured, which is not uncommon in real clinical
trials. Basic parameters were varied in the following ranges: (i)
number of units $N\in \{100,200,300,400,500\}$; (ii) number of time
steps $T\in \{4,6,8,10,12\}$; (iii) time lag $l\in \{0,1,2,3,4\}$;
effect size $\tau_l\in \{0.01,0.02,0.03,0.04,0.05\}$. The
performance of the methods was examined when one of $N$, $T$, and $l$
is changed while the other two are fixed at the medians of their
ranges, respectively. For example, we
increased $N$ from 100 to 500 while keeping $T=8$ and $l=2$. In these
simulations, $\tau_l$ was set to $0$ and $0.03$ to investigate the level and power of the methods, respectively.
 One more simulation was created to study the power as the effect size
 $\tau_l$ varies, in which we keep the first three parameters at their
 medians $(N=300,T=8, l=2)$ and increase $\tau_l$ from 0.01 to
 0.05.

The upper panels of \Cref{fig:error_power} show that all the methods
control the type I errors at any number of units, time steps and time
lags. The lower panel shows that our methods are more
powerful than Bonferroni's method in all the simulations. MCRTs+Z is
slightly more powerful than MCRTs+F in all experiments. This shows that
the weighted z-score is more effective than Fisher's combination
method for MCRTs. Panels (d) and (g) show that our methods
outperform Bonferroni's method by a wider margin as the sample size or the effect
size increases. Panel (e) establishes the same observation for trial
length, which can be explained by the fixed sample size (so fewer
units start treatment at each step as trial length increases). This is
not an ideal scenario for applying Bonferroni's method, as the
individual tests have diminishing power. Finally,
panel (f) shows that all the methods have smaller power as the lag size increases, and our methods are particularly powerful when the time lag is small. This is because there are fewer permutation tests available for larger time lags. Moreover, by only including the units that are treated after a large time lag, the control groups in the permutation tests are small.
Overall, the reduced sample size in MCRTs may not be as damaging
as it might first appear because every outcome relevant to
lag $l$ effect is still used in at least one test in
MCRTs.


\subsection{Stepped-wedge design: misspecified mixed-effects models }\label{sect:sim_2_mixed}

We next investigate the finite-sample properties
of confidence intervals obtained by inverting conditional
randomization tests (see \Cref{appenx:invert}
for more details). We compare its efficiency and
robustness with mixed-effects models for estimating lagged treatment effects.

\begin{figure}[t]
  \centering
  \includegraphics[width=0.95\linewidth]{./Figures/exp_interact}
  \caption{Performance of MCRTs+F, MCRTs+Z and the mixed-effects model: coverage rates and lengths of confidence intervals (CIs) under the outcome generating processes with no interaction effect and with three different types of covariate-and-time interactions. } 
  \label{fig:interaction}
  \centering
\end{figure}

The treatment generating process was kept
the same as above. Simulation parameters are set as $N=200$,
$T=8$ and lagged effects $(\tau_0,\dotsc,\tau_7) =
(0.1,0.3,0.6,0.4,0.2,0,0,0)$. The treatment effect is gradually realized and then decays to 0 over time.
 For every $i\in [N]$ and $t=0,1,\dotsc,T$, we generate the outcome $Y_{it}$ using a mixed-effects model,
\[
Y_{it} = \mu_i + 0.5(1-0.1\cdot1_{\{m\neq 0\}})(X_i+t) + 0.1 f_m(X_i+t) + \sum_{l=0}^{7}1_{\{A_i-t =l \}}\tau_l+\epsilon_{it},
\]
where $\mu_i\sim \mathcal N(0,0.25)$, $X_i\sim \mathcal N(0,0.25)$,
$\epsilon_{it}\sim \mathcal N(0,0.1)$ and the unit-by-time interaction $f_m(X_i+t)=0$ if $m = 0$ (no interaction),
$(X_i+t)^2$ if $m = 1$ (quadratic interaction), $2\exp[
(X_i+t)/2]$ if $m = 2$ (exponential interaction), and $5\tanh(X_i+t)$
if $m$=3 (hyperbolic tangent interaction).
We tasked our methods (MCRTs+F and MCRTs+Z) and a mixed-effects model
described below to construct valid 90\%-confidence intervals (CIs) for the lagged effects $\tau_0,\dotsc
\tau_4.$
The mixed-effects model takes the form
\begin{equation}\label{equ:mixed_model}
Y_{it} = \beta_{0i} + \beta_1 x_i + \beta_2 t +
\sum_{l=0}^{7}\xi_{l}1_{\{A_i-t =l \}}+ \epsilon_{it},
\end{equation}
where $\beta_{0i}$ is a random unit effect, $\beta_1, \beta_2$ are fixed
effects and $\xi_{0},\dotsc,\xi_{7}$ are lagged effects, and were
fitted using the R-package \texttt{lme4}
\citep{Douglas2015bates}. Recently, \citet{kenny2021analysis} proposed mixed-effects models that can leverage the shapes of time-varying treatment effects. Their models are given by specifying some parametric effect curves with the help of basis functions (e.g. cubic spline).
Besides the unit-by-time interaction, the model \eqref{equ:mixed_model} is already correctly specified for modelling the treatment effects and other parts of the outcome generating process above.
Excluding some of the lagged effect parameters $\xi_5,\dotsc,\xi_7$
from the model may change the effect estimates but not the validity of its CIs.

As the unit-by-time interactions are unknown
and fully specifying them would render the model unidentifiable, they
are typically not considered in mixed-effect models. If the
stepped-wedge trial is randomized at the cluster levels, one can
introduce cluster-by-time interaction effects in a mixed-effects
model; see \citet[Equation 3.1]{ji2017randomization} for an
example. However, the cluster-by-time interactions are not exactly the
same as the interaction between time and some covariates of the
units. It depends on if the interaction varies within each cluster and the clusters are defined by the covariates which interact with time.

The upper panels in \Cref{fig:interaction} show that the CIs of the
mixed-effects model only achieve the target coverage rate of $90\%$
for all the lagged effects when there is no unit-by-time interaction
(panel a). In contrast, the CIs
of our methods fulfil the target coverage rate of $90\%$ in all
scenarios. The lower panels show that the valid CIs given by our
methods have reasonable
lengths between 0.05 and 0.10 for covering the true lagged effects
$(\tau_0,\dotsc,\tau_4) = (0.1,0.3,0.6,0.4,0.2)$. Finally, panel
(e) shows that when the mixed-effects model is indeed specified correctly,
it gives valid CIs that are narrower than our nonparametric tests.

\subsection{Simulation II: general interference}\label{sect:sim_interference}

We next assess the validity and power of the static method and the MCRTs(+F) method from \Cref{sec:spillover}.
Our simulation has three parameters: sample size $N\in \{100,150,200,250,300\}$, proportion of treated units $P_N \in \{0.1,0.2,0.3,0.4,0.5\}$ and signal strength $\beta \in \{0,0.5,1.0,1.5,2.0\}$. For $i\in [N],$ unit $i$'s location $\bm X_i = (X_{i1},X_{i2})$ is drawn independently from a mixture of bivariate normal distributions,
\begin{align*}
& X_{i1}\sim 0.5\mathcal{N}(50,100) + 0.3\mathcal{N}(30,56.25) + 0.2 \mathcal{N}(40,56.25),\\
& X_{i2}\sim 0.5\mathcal{N}(50,100) + 0.3\mathcal{N}(60,56.25) + 0.2 \mathcal{N}(20,56.25).
\end{align*}
Let $D_{ij}$ denote the Euclidean distance between units $i$ and $j$. The outcome of each unit $i$ is given by
\[
Y_i = 4Z_i +  \sum_{j\in [N]\setminus \{i\}}\sum_{k\in [5]} 1_{\{  D_{ij} \in  ( k-1,k] \}} Z_j \beta \tau_k + E_i,
\]
where  $E_i\sim \mathcal N(0,1)$ and $(\tau_1,\tau_2,\tau_3,\tau_4,\tau_5) = (2.0,1.6,1.2,0.8,0.4)$ is a decreasing spillover effect vector. As we vary one of the parameters  $N,P_N$ and $\beta$, we fix the other two at the median of their respective ranges.   In all the CRTs, we use the difference-in-means statistic and fix the number of permutations at 1000.

\begin{figure}[ht]
\vspace{-5pt}
  \centering
    \includegraphics[width = 0.8\textwidth]{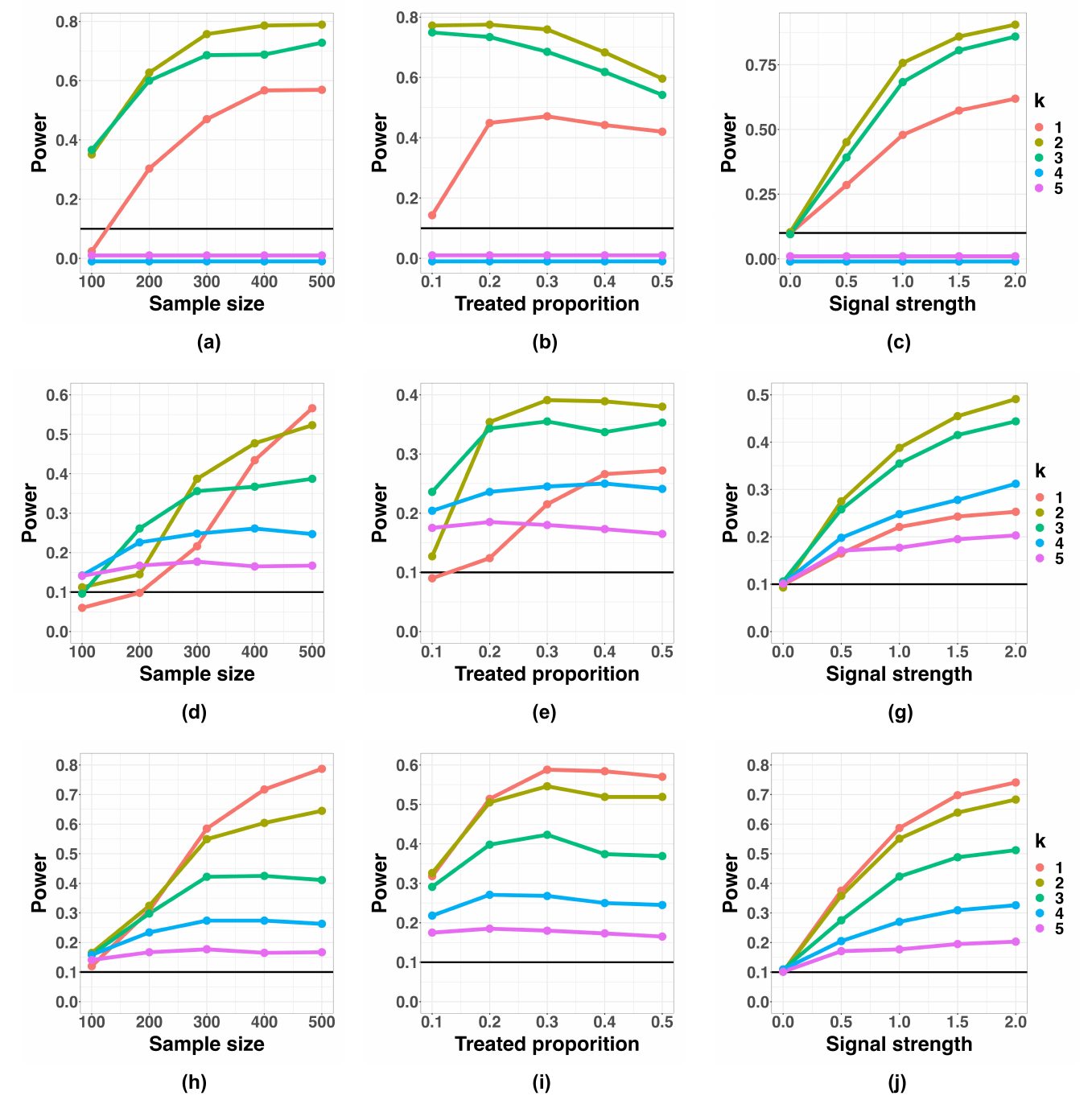}
    \caption{Performance of the static method and the MCRTs(+F) method in testing the hypotheses $H^{(k)}$ for $k\in [5]$. The upper panels use the p-values from the static method.
  The middle panels use the p-values from MCRTs. The lower panels combine the p-values from
  MCRTs using Fisher's method, denoted by MCRTs+F. The results are averaged over 1000 independent runs.}
  \label{fig:interference_plot}
  \vspace{2pt}
\end{figure}
The results of all the methods are reported in \Cref{fig:interference_plot}.
In the static method, we let the focal units be the minimum 3-cover of all the units. The top panels show that the static method is powerful in testing the hypothesis of no spillover effect between units with distance more than 3, but has no power for longer distances because no units can be found for the randomization test (similar to \Cref{fig:interference_0}(d) in the real data example). The middle panels show that the power of our adaptive MCRTs is not a simple decreasing function of the distance $k$. The bottom panels show that this can be addressed by combining the tests at different distances (for example, we use Fisher's method that combines $P^{(k)},\dotsc, P^{(5)}$ to test $H^{(k)}$) because the hypotheses are nested and the p-values are nearly independent. Panels (c,g,f) verify that the CRTs control the type I error rate at the significance level $\alpha=0.1$.


%

\clearpage

\section{Multiple conditional randomization tests}
\label{sec:multiple-crt}

We next generalize the theory of sequential CRTs in
\Cref{sec:sequential-crts}. For any subset of the CRTs $\mathcal{J}
\subseteq [K]$, we define
the \emph{union}, \emph{refinement} and \emph{coarsening} of their
conditioning sets  as
\[
  \mathcal{R}^{\mathcal{J}} = \bigcup_{k \in \mathcal{J}}
  \mathcal{R}^{(k)}, \quad
  \underline{\mathcal{R}}^{\mathcal{J}} = \left\{ \bigcap_{j \in \mathcal{J}}
    \mathcal{S}_{\bm z}^{(j)} : \bm z \in
    \mathcal{Z}\right\}, \quad \text{and} \quad
  \overline{\mathcal{R}}^{\mathcal{J}} = \left\{\bigcup_{j \in \mathcal{J}}
    \mathcal{S}^{(j)}_{\bm z}: \bm z \in
    \mathcal{Z}\right\}.
\]
Let $\mathcal{G}^{\mathcal{J}}$ be the $\sigma$-algebra generated
by the sets in $\mathcal{R}^{\mathcal{J}}$, so
$\mathcal{G}^{\mathcal{J}} = \sigma\big(\mathcal{R}^{\mathcal{J}} \big).$
Similarly, the refinement and
coarsening $\sigma$-algebras are defined as
$\underline{\mathcal{G}}^{\mathcal{J}} =
\sigma\big(\underline{\mathcal{R}}^{\mathcal{J}} \big)$ and
$\overline{\mathcal{G}}^{\mathcal{J}} =
\sigma\big(\overline{\mathcal{R}}^{\mathcal{J}} \big)$.

\begin{theorem}\label{thm:general-3}
  Suppose the following two conditions are satisfied for all $j, k\in [K]$, $j \neq k$.
  \begin{equation}
    \label{eq:nested}
    \underline{\mathcal{R}}^{\{j,k\}}\subseteq \mathcal{R}^{\{j,k\}},
  \end{equation}
  \begin{equation} \label{eq:cond-indep-3}
    \left( t^{(j)}(\bm Z, \bm W)\right)_{j \in \mathcal{J}} \text{ are independent given }
    \underline{\mathcal{G}}^{\mathcal{J}} , \bm W \text{ for
      all }
    \mathcal{J} \subseteq [K],
  \end{equation}
  where conditioning  on $\underline{\mathcal{G}}^{\mathcal{J}}$ means
  conditioning on $\bm Z\in \mathcal{S}$ for any set $\mathcal{S}\in
  \underline{\mathcal{R}}^{\mathcal{J}}$.
  Under Assumption  \ref{assump:randomization}, for any $ \alpha^{(1)},\dotsc,\alpha^{(K)} \in [0,1]$,
  \begin{equation}\label{equ:dominance-conditional-3}
    \begin{split}
      \mathbb{P}\left\{P^{(1)}(\bm Z, \bm W) \leq \alpha^{(1)},\dotsc, P^{(K)}(\bm
      Z, \bm W) \leq \alpha^{(K)} \mid \overline{\mathcal{G}}^{[K]},
      \bm W\right\}
      \leq \prod_{k=1}^{K}\alpha^{(k)}.
    \end{split}
  \end{equation}
  In consequence, \eqref{equ:dominance-3} holds for any $ \alpha^{(1)},\dotsc,\alpha^{(K)} \in [0,1]$.
\end{theorem}

The proof of \Cref{thm:general-3} is not straightforward and can be found in \Cref{sec:app-proof}.
Comparing to \Cref{prop:nested}, it is allowed that  $\mathcal{S}^{(1)}_{\bm z}
\subseteq \mathcal{S}^{(2)}_{\bm z}$ for some $\bm z$ but
$\mathcal{S}^{(1)}_{\bm z^{*}} \supseteq \mathcal{S}^{(2)}_{\bm
  z^{*}}$ for another $\bm z^{*} \neq \bm z$.
When $K=2$, condition \eqref{eq:cond-indep-3} is equivalent to
\begin{equation}\label{equ:indep_Tz}
  t^{(1)}(\bm Z, \bm W) \independent t^{(2)}(\bm Z, \bm W) \mid \bm Z \in \mathcal{S}^{(1)}_{\bm
    z} \cap \mathcal{S}^{(2)}_{\bm z}, \bm W, \quad \forall \bm z
  \in \mathcal{Z},
\end{equation}
\Cref{equ:indep_Tz} is satisfied by our sequential CRTs in \Cref{sec:sequential-crts}, where  $t^{(j)}(\bm Z, \bm W)$  is a function of $\bm V^{[j]}$ and $\bm W$
and the condition  $\bm Z \in \mathcal{S}_{m}^{(k)}$ fixes all the randomness in $\bm V^{(j)}$ for  $j<k$.

\section{Discussion}
\label{sec:discussion}


In this article, we demonstrated how to generate
nearly independent and powerful p-values through a sequential construction of multiple conditional randomization tests. Our idea led to new methods for testing lagged treatment effects in stepped-wedge randomized trials and spillover effects in spatial experiments.
Our methods are model-free and thus robust against model misspecification, as confirmed by our numerical experiments.




Our theory assumes a randomized experiment. In observational studies,
matching by observed confounders is often used to mimic a randomized
experiment using observational study  \citep{rosenbaum2002observational}. This heuristic has been more
formally investigated recently
\citep{pimentel2022covariate,guo2023statistical}. It would be
interesting to study if the theory presented here can be extended to
randomization inference for observational studies.


In another line of research, it has been shown that a
two-sample permutation test or randomization test with a carefully
constructed test statistic is asymptotically valid for testing
equality of a summary parameter between the samples
\citep{Chung2013,bertanha2023permutation}, average treatment effects
\citep{cohen2020gaussian,fogarty2021prepivoted,zhao2021covariate},
``bounded'' null hypotheses \citep{caughey2021randomization}, or
quantiles of individual treatment effects
\citep{caughey2021randomization}. It may be possible to
combine these techniques and the sequential construction in this
article to composite null hypotheses and obtain nearly independent
p-values.

It would also be interesting to consider more complex design spaces
and their topological structures in future works. For example, \citet{auerbach2022testing} recently introduced a randomization
test for the hypothesis that two networks are drawn from the same
random graph model. As another example, unlike in the stepped-wedge
design, a unit can switch to any treatment arm at
any time in general cross-over designs or panel experiments
\citep{bojinov2021panel,shahn2023formal}. The theory outlined in this article in principle applies to any experimental design, but some non-trivial adaptions are needed for each specific setting as shown in the examples above.


\bibliographystyle{plainnat}
\bibliography{main_mcrts}

\appendix

\clearpage

\begin{appendices}
 \section{Method of combining p-values}
\label{sec:combine}

A remaining question is how to combine the permutation tests obtained
from \Cref{alg:mcrts} (MCRTs) efficiently for testing spillover effects in \Cref{sec:swd}. \citet{heard2018choosing}
compared several p-value combination methods in the literature. By
recasting them as likelihood ratio tests, they demonstrated that the
power of a combiner crucially depends on the distribution of the p-values under the
alternative hypotheses. In large samples, the behaviour of permutation
tests is well-studied in the literature. \citet[Theorem
15.2.3]{lehmann2006testing} showed that if the test statistics in a
permutation test converges in distribution, the permutation
distribution will converge to the same limiting distribution in
probability. These results form the basis of our investigation of
combining CRTs.


Suppose that $K$ permutation p-values $P^{(1)},\dotsc, P^{(K)}$ are
obtained from \Cref{alg:mcrts}. Suppose the $k$th test statistic,
when scaled by $\sqrt{N}$, is asymptotically normal so that it
converges in distribution to $T_{\infty}^{(k)} \sim \mathcal{N}
(\tau_l, V_{\infty}^{(k)})$
under
the null hypothesis ($\tau_l = 0$) and some
local alternative hypothesis ($\tau_l = h/ \sqrt{N}$) indexed by
$h$. The limiting distributions have the same mean but different
variances. Suppose that we define the p-value likelihood ratio as the product of the alternative p-value distributions for all $k$ divided by the product of the null p-value distributions for all $k$.
Since the permutation p-values are standard uniform,
it is easy to verify that the logarithm of this p-value likelihood ratio is
proportion to a weighted sum of $Z$-scores, $T_{\infty} =
\sum_{k=1}^{K}   w_{\infty}^{(k)} \Phi^{-1}(P^{(k)})$. This motivates a
weighted Stouffer's method \citep{stouffer1949american} for combing the
p-values. The weights are non-negative and sum to one, thus $T_{\infty} \sim \mathcal{N}(0,1)$ if the p-values are independent. Then we can
reject $\bigcap_{k\in[K]} H_{0}^{(k)}$ if  $\Phi(T_{\infty})\leq
\alpha$. This test is shown to be uniformly most powerful for
all $h$ in the normal location problem \citep{heard2018choosing}.

To formalize the discussion above, we consider the difference-in-means
statistics as an example. Following the notation in \Cref{alg:mcrts}
(see lines 6 and 11), the treated group $\mathcal{I}_1^{(k)}$ for the
$k$th CRT crosses
over at time $k$ and the control group $\mathcal{I}_0^{(k)}$ crosses
over time  $t\in \mathcal{T}^{(k)}\setminus \{k\}$, where
$\mathcal{T}^{(k)}\setminus \{k\}$ is a subset of $\mathcal{C}_j$ that
collects some time points after $k+l$.

Let $N_1^{(k)} = |\mathcal{I}_1^{(k)}|$, $N_0^{(k)} = |\mathcal{I}_0^{(k)}|$ and $N^{(k)} = N_0^{(k)} + N_1^{(k)}$
Let $A_i^{(k)} = 1 \text{ or } 0 $ denote unit $i \in
\mathcal{I}_0^{(k)} \cup \mathcal{I}_1^{(k)}$ starting the treatment
at time $k$ or after $k+l$, i.e.,
$ \bm Z_{i,k+l} = (\bm 0_{k-1},1,\bm 0_{l}) \text{ or }  \bm
0_{k+l}$. To simplify the exposition in the results below, we use the abbreviations
\[
Y_{i}^{(k)} = Y_{i,k+l},\ \  Y_{i}^{(k)}(1) = Y_{i,k+l}(\bm 0_{k-1},1,\bm 0_{l}) \ \text{ and }\  Y_{i}^{(k)}(0) = Y_{i,k+l}(\bm 0_{k+l})
\]
We further assume that the $N$ units are i.i.d draws from a super-population model, and
every unit's potential outcome $Y_i^{(k)}(a)$ is a
random copy of some generic $Y^{(k)}(a)$ for $a\in\{0,1\}$ and $k\in
{K}$.



\begin{proposition}\label{thm:combine}
Suppose that the $K$ permutation tests in \Cref{alg:mcrts} use the difference-in-means statistics.
We assume that for every $k\in [K]$ and $a\in \{0,1\}$,
\begin{enumerate}
\item $Y^{(k)}(a)$ has a finite and nonzero variance;
\item $N_1^{(k)}/N$ and $N_0^{(k)}/N$ are fixed as $N\rightarrow\infty$;
\end{enumerate}
The weighted $Z$-score combiner $T_{\infty}$ is then given with weights $w_{\infty}^{(k)} = \sqrt{\frac{\Lambda^{(k)}}{ \sum_{j=1}^{K} \Lambda^{(j)} }},$ where
\begin{equation}
  \label{eq:lambda-k}
\Lambda^{(k)} = \left(\frac{N}{N_0^{(k)}} \Var\big[Y^{(k)}(1)\big] + \frac{N}{N_1^{(k)}} \Var\big[Y^{(k)}(0)\big] \right)^{-1}  \end{equation}
is the inverse of the asymptotic variance $V_{\infty}^{(k)}$  of the statistics in the $k$-th test.
\end{proposition}

The weights in \cref{eq:lambda-k} are obtained by inverting the
asymptotic variance of the difference-in-means statistics; see
\citet[Section 6.4]{imbens2015causal} and \citet{li2017general}.
We can estimate $\Var\big[Y^{(k)}(a)\big]$ in
$\Lambda^{(k)}$ consistently using the sample variance of $Y_i^{(k)},\forall i\in \mathcal I_a^{(k)}$. We denote the estimator of $\Lambda^{(k)}$  by $\hat \Lambda^{(k)}$.
The empirical version of $T_{\infty}$ is given by
\begin{equation}\label{equ:weighting}
  \hat{T} = \sum_{k=1}^{K}  \hat{w}^{(k)} \Phi^{-1}(P^{(k)}) \quad  \text{where}\quad
   \hat{w}^{(k)}= \sqrt{\frac{\hat \Lambda^{(k)}}{ \sum_{j=1}^{K}\hat \Lambda^{(j)}}}.
\end{equation}

\begin{proposition}\label{lemma:p_hat}
        Suppose that the p-values $P^{(1)},\dotsc, P^{(K)}$  from \Cref{alg:mcrts} are valid, the combined p-value $\hat{P}(\bm Z, \bm W) = \Phi(\hat{T} )$ is also valid such that under the null hypotheses,
\[
\mathbb P\{     \hat{P}(\bm Z, \bm W) \leq \alpha               \}\leq \alpha.
\]
\end{proposition}

In general, improving the statistical power in combining multiple CRTs is an
independent task after choosing an efficient test statistics. 
Different test statistics may have different asymptotic
distributions and the optimal p-value combiners (if exist) may be
different. Nonetheless, a key insight from the discussion above is
that we should weight the CRTs appropriately, often according to their
sample size. 

\section{Confidence intervals from randomization tests}\label{appenx:invert}

Here we describe how to invert (a combination of) permutation tests
and the complexity involves; see also \citet[Section 3.4]{ernst2004permutation} for an introduction.
Consider a completely randomized experiment with treated outcomes $(Y_1,\dotsc,Y_m)$ and control outcomes $(Y_{m+1},\dotsc,Y_{m+n})$.

Consider the constant effect null hypothesis
\[
H_0: Y_{i}(1) = Y_{i}(0) +\Delta,\forall i \in [m+n].
\]
We can implement a  permutation test for $H_0$ by testing the zero-effect hypothesis on the shifted outcomes $\bm Y_{\Delta} = (Y_1-\Delta,\dotsc,Y_m-\Delta,Y_{m+1},\dotsc,Y_{m+n})$.
To construct a confidence interval for the true constant treatment effect $\tau$, we can consider all values of $\Delta$ for which we do not reject $H_0$. We define the left and right tails of the randomization distribution of $t(\bm Z^*,\bm Y_{\Delta})$ by
\[
P_1(\Delta) = \mathbb P^*\{t(\bm Z^*,\bm Y_{\Delta})\leq t(\bm Z,\bm Y_{\Delta})\} \text{ and } P_2(\Delta) = \mathbb P^*\{t(\bm Z^*,\bm Y_{\Delta})\geq t(\bm Z,\bm Y_{\Delta})\},
\]
where $\bm Z = [\bm 1_{m}/m,-\bm 1_{n}/n]$ is observed assignment, $\bm Z^*$ is a permutation of $\bm Z$, $T$ is the test statistics, e.g., difference-in-means, $t(\bm Z,\bm Y_{\Delta}) = \bm Z^{\top} Y_{\Delta}$. The complexity of computing $P_1(\Delta)$ and $P_2(\Delta)$
with $b$ different permutations is $O(mb+nb)$.  The two-sided $(1-\alpha)$-confidence interval for $\Delta$ is given by
$[\Delta_L,\Delta_U] := \left[ \min_{P_2(\Delta) >\alpha/2}\Delta, \max_{P_1(\Delta) >\alpha/2}\Delta \right].$
We use a grid search to approximately find the minimum and maximum
$\Delta$ in $[\Delta_L,\Delta_U]$. Since $P_1(\Delta)$ and
$P_2(\Delta)$ are monotone functions of $\Delta$, we can also consider
obtaining the optimal $\Delta$'s via a root-finding numerical method
\citep[see e.g.][]{garthwaite1996confidence}.

Inverting a combination of permutation tests can be done similarly. For example, by searching the same $\Delta$'s for every permutation test,
the lower confidence bound $\Delta_L$ is given by the minimum $\Delta$ under which the p-value of the combined test (e.g. $\hat P $ in \Cref{lemma:p_hat}) is larger than $\alpha/2$.
The complexity of inverting a combined test scales linearly in terms of the number of tests.
The total complexity is manageable as long as the numbers of units and permutations are not very large at the same time.

If we have covariates information in the dataset, we can consider fitting a linear regression model (with basis functions for nonlinearity). We can compute the matrix inversion in the least-squares solution once, then update the solution by shifting the outcome vector with different $\Delta$'s. Inverting the test returns an interval with coverage probability approximately equal to  $1-\alpha$. When the probability is slightly below $1-\alpha$, we can decrease $\alpha$ by a small value (e.g. 0.0025) gradually and reconstruct the interval based on the previously searched $\Delta$'s. We stop if we obtain an interval with coverage probability above $1-\alpha$.


\section{Pseudocode for multiple conditional randomization tests (MCRTs)}\label{appendix:randomization}



See \Cref{alg:mcrts}.

\begin{algorithm}[h]\label{alg:mcrts}
    \caption{Multiple conditional randomization tests (MCRTs) for
      testing lag-$l$ treatment effect in stepped-wedge randomized
      controlled trials}\label{alg:mcrts}
    \hspace*{\algorithmicindent}
    \begin{algorithmic}[1]
    \vspace{-5pt}
   \State    \textbf{Input:} Number of units $N$, Number of time steps $T$, Time lag $l$, Outcomes $\bm{Y} = (Y_{it}:i\in [N],t\in [T])$,  Assignments $\bm Z = (Z_{it}:i\in [N],t\in [T])$, Statistics $T$.
  \State \textbf{Initialization:} $J \gets \min(l+1,T-l-1)$ and $\mathcal{I}_t \in \{i\in [N]: Z_{it}=1  \}, \forall t\in [T]$
  \vspace{2pt}
        \For {$j \in [J]$}                              \Comment{Divide the time steps $[T]$ into $J$ subsets }
\vspace{1pt}
\State
\vspace{1pt} $t \gets j  $, $\mathcal{C}_j \gets \{t\}  $
        \While{$t+l+1 \leq T$}
\vspace{1pt}
\State
\vspace{1pt} $t \gets t  + l +1$, $\mathcal{C}_j \gets \mathcal{C}_j \cup\{t\}  $
\EndWhile
\EndFor
\For {$j \in [J]$}
\For {$k \in \mathcal{C}_j$}            \Comment{Define the $k$-th permutation test}
\State
$\mathcal{T}^{(k)}\gets  \big\{t\in \mathcal{C}_j: t\geq k \big\}$\vspace{1pt}
\State
\vspace{1pt}
  $\mathcal{I}_1^{(k)} \gets \mathcal{I}_{k}$, $\bm Y_{\text{treated}}^{(k)} \gets \big\{Y_{i,k+l}:i \in \mathcal{I}_1^{(k)}\big\}$
\vspace{1pt}
 \State
\vspace{1pt}
   $\mathcal{I}_0^{(k)} \gets (\bigcup_{t\in \mathcal{T}^{(k)} }        \mathcal{I}_t\big)      \setminus \mathcal{I}_{k}$, $\bm Y_\text{control}^{(k)} \gets \big\{Y_{i,k+l}:i \in \mathcal{I}_0^{(k)}\big\}$                 
\vspace{1pt}
\State
\vspace{1pt} $P^{(k)}(\bm Z,\bm Y) \gets  \text{Permutation Test} \big(\bm Y_{\text{treated}}^{(k)},\bm Y_\text{control}^{(k)}; T \big) $

\EndFor
\EndFor
\vspace{1pt}
\State  \textbf{Output:} P-values $\big\{P^{(k)}:=P^{(k)}(\bm Z,\bm Y)\big\}$
        \end{algorithmic}
    
\end{algorithm}

\section{Technical Proofs}
\label{sec:app-proof}

\subsection{Proof of \Cref{thm:valid}}\label{sec:proof-crefl-2}

\begin{proof}
  We first write the p-value in \cref{eq:p-value} as a probability
  integral transform. For any fixed $\bm z \in \mathcal{Z}$, let
  $F_{\bm z}(\cdot;\bm W)$ denote the distribution function of $T(\bm Z, \bm W)$ given $\bm W$ and $\bm Z \in \mathcal{S}_{\bm
    z}$. Given $\bm Z \in \mathcal{S}_{\bm
    z}$ (so $\mathcal{S}_{\bm Z} = \mathcal{S}_{\bm z}$ by \Cref{def:CRT}), the p-value can be written
  as
$    P(\bm Z, \bm W) = F_{\bm z}(T_{\bm z}(\bm Z, \bm W);\bm W).$
For any random variable $T$ and its distribution function $F(t) = \P(T \le t)$ we have
  $\mathbb{P}(F(T) \leq \alpha) \leq \alpha$ for all $0 \leq \alpha
  \leq 1$; see \Cref{lem:pit} below. 
  By the law of total probability,
    \begin{align*}
    \mathbb{P}\left\{ P(\bm Z, \bm W)\leq \alpha \mid V, \bm W \right\} &=
                                                                       \sum_{m=1}^M \mathbb{P}\left\{ P(\bm Z, \bm W)\leq \alpha \mid \bm
                                                                       Z \in \mathcal{S}_m, \bm W \right\} 1_{\{V=m\} } \\
                                                                     &\leq
                                                                       \sum_{m=1}^M
                                                                       \alpha
                                                                       1_{\{\bm
                                                                       Z \in \mathcal{S}_m\} }
                                                                       =
                                                                       \alpha.
  \end{align*}
  Marginalizing over the potential outcomes schedule $\bm W$, we attain the last claim in the theorem.
\end{proof}



\begin{lemma} \label{lem:pit}
  Let $T$ be a random
  variable and $F(t) = \P(T \le t)$ be its distribution function. Then
  $F(T)$ has a distribution that stochastically
  dominates the uniform distribution on $[0,1]$.
\end{lemma}
\begin{proof}
  Let $F^{-1}(\alpha ) = \sup \{ t\in
  \mathbb{R} \mid F(t)\leq \alpha
  \}$. We claim that $\mathbb{P}\{ F(T)\leq \alpha \} = \mathbb{P}\{T
  < F^{-1}(\alpha) \}$; this can be verified by considering whether
  $T$ has a positive mass at $F^{-1}(\alpha)$ (equivalently, by
  considering whether $F(t)$ jumps at $t = F^{-1}(\alpha)$). Then, we have
\[
 \mathbb{P}\left\{ F(T)\leq \alpha \right\} = \mathbb{P}\left\{T < F^{-1}(
      \alpha) \right\} =
    \lim_{t \uparrow F^{-1}(
      \alpha)} F(t) \leq \alpha,
\]
using the fact that $F(t)$ is non-decreasing and right-continuous.
\end{proof}

\subsection{Proof outline of \Cref{thm:general-3}}
\label{sec:hasse-diagr-cond}

We now outline a proof of \Cref{thm:general-3}.  We start with the
following observation. (All proofs of the technical results can be
found in the following sections.)

\begin{lemma} \label{lemma:disjoint-or-nested}
  Suppose \cref{eq:nested} is satisfied. Then for any $\mathcal{J}
  \subseteq [K]$ and $\mathcal{S},\mathcal{S}' \in
  \mathcal{R}^{\mathcal{J}}$, the sets $\mathcal{S}$ and
  $\mathcal{S}'$ are either disjoint or nested, that is,
\[
\mathcal{S}\cap \mathcal{S}'\in \{ \emptyset, \mathcal{S},\mathcal{S}'\}.
\]
  Furthermore, we have
  $\underline{\mathcal{R}}^{[K]} \subseteq \mathcal{R}^{[K]}$ and
  $\overline{\mathcal{R}}^{[K]}\subseteq \mathcal{R}^{[K]}$.
\end{lemma}

The sets in $\mathcal{R}^{[K]}$ can be partially ordered by set
inclusion. This induces a graphical structure on $\mathcal{R}^{[K]}$:

\begin{definition}\label{def:hasse}
  The \emph{Hasse diagram} for $\mathcal{R}^{[K]}  =
  \{\mathcal{S}_{\bm z}^{(k)}:\bm
  z \in \mathcal{Z}, k \in [K]\}$ is a graph where each node in the
  graph is a set in $\mathcal{R}^{[K]}$ and a
  directed edge $\mathcal{S} \rightarrow \mathcal{S}'$ exists
  between two distinct nodes $\mathcal{S},\mathcal{S}' \in
  \mathcal{R}^{[K]}$ if $\mathcal{S} \supset \mathcal{S}'$ and there is no
  $\mathcal{S}'' \in \mathcal{R}^{[K]}$ such that $\mathcal{S}
  \supset \mathcal{S}''\supset \mathcal{S}'$.
\end{definition}

It is straightforward to show that all edges in the Hasse diagram for
$\mathcal{R}^{[K]}$ are directed and this graph has no cycles. Thus,
the Hasse diagram is a directed acyclic graph.

For any node
$\mathcal{S} \in \mathcal{R}^{[K]}$ in this graph, we can further
define its parent set as
\[
\text{pa}(\mathcal{S}) = \{\mathcal{S}^{'} \in
\mathcal{R}^{[K]} : \mathcal{S}' \to \mathcal{S}\},
\]
child set as
$\text{ch}(\mathcal{S}) = \{\mathcal{S}^{'} \in \mathcal{R}^{[K]} :
\mathcal{S} \to \mathcal{S}'\}$, ancestor set as $\text{an}(\mathcal{S})
= \{\mathcal{S}'\in \mathcal{R}^{[K]} : \mathcal{S}' \supset
\mathcal{S} \}$, and descendant set as $\text{de}(\mathcal{S}) =
\{\mathcal{S}'\in \mathcal{R}^{[K]} : \mathcal{S} \supset \mathcal{S}'
\}$.
We note that one conditioning set $\mathcal{S} \in \mathcal{R}^{[K]}$
can be used in multiple CRTs. To fully characterize this structure, we
introduce an additional notation.

\begin{definition}
  For any $\mathcal{S} \in \mathcal{R}^{[K]}$,
  let $\mathcal{K}(\mathcal{S}) = \{k \in [K]: \mathcal{S} \in
  \mathcal{R}^{(k)}\}$ be the collection of indices such that
  $\mathcal{S}$ is a conditioning set in the corresponding
  test. Furthermore, for any collection of conditioning sets
  $\mathcal{R}\subseteq \mathcal{R}^{[K]}$, denote
  $\mathcal{K}(\mathcal{R}) = \cup_{\mathcal{S} \in \mathcal{R}}
  \mathcal{K}(\mathcal{S})$.
\end{definition}

\begin{lemma}\label{lem:hasse-structure}
  Suppose \cref{eq:nested} is satisfied. Then for any $\mathcal{S}
  \in  \mathcal{R}^{[K]}$, we have
  \begin{enumerate}[label=(\roman*)]
  \item \label{equ:hasse1} If $\text{ch}(\mathcal{S})\neq
    \emptyset$, then
    $\text{ch}(\mathcal{S})$ is a partition of
    $\mathcal{S}$;
  \item \label{equ:hasse2} $\{\mathcal{K}(\text{an}(\mathcal{S})),
    \mathcal{K}(\mathcal{S}), \mathcal{K}(\text{de}(\mathcal{S}))\}$
    forms a partition of $[K]$.
  \item \label{equ:hasse3} For any $\mathcal{S}' \in \text{ch}(\mathcal{S})$,  $\mathcal{K} ( \text{an}(\mathcal{S}')) = \mathcal{K} (
    \text{an}(\mathcal{S})\cup \{\mathcal{S}\})$ and
    $\mathcal{K} (\{\mathcal{S}'\}\cup  \text{de}(\mathcal{S}'))
    = \mathcal{K} ( \text{de}(\mathcal{S}))$.
  \end{enumerate}
\end{lemma}

Using \Cref{lem:hasse-structure}, we can prove the following key lemma
that establishes the conditional independence between the p-values.

\begin{lemma} \label{lem:p-value-cond-indep}
  Suppose \cref{eq:nested} and \cref{eq:cond-indep-3} are
  satisfied. Then for any $\mathcal{S}\in \mathcal{R}^{[K]}$ and $j\in
  \mathcal{K}(\mathcal{S})$,
  \begin{equation*}\label{equ:nested_indep}
    P^{(j)}(\bm Z, \bm W) \independent \left( P^{(k)}(\bm Z, \bm W) \right)_{k\in
  \mathcal{K}(\text{an}(\mathcal{S})\cup \{\mathcal{S}\}) \setminus
  \{j\}} \mid \bm Z\in \mathcal{S},\bm W.
  \end{equation*}
\end{lemma}

Finally, we state a general result based on the above Hasse diagram,
from which \Cref{thm:general-3} almost immediately follows.

\begin{lemma}\label{lemma:general_lemma}
  Given conditions \cref{eq:nested} and \cref{eq:cond-indep-3}, we
  have, for any $\mathcal{S} \in \mathcal{R}^{[K]}$,
  \begin{equation}\label{equ:general_lemma}
    \begin{split}
      &~\mathbb{P}\left\{P^{(1)}(\bm Z, \bm W) \leq \alpha^{(k)},\dotsc, P^{(K)}(\bm
        Z, \bm W) \leq \alpha^{(K)} \mid \bm Z \in \mathcal{S}, \bm
        W\right\} \\
      \leq&~ \mathbb{P}\left\{P^{(k)}(\bm Z, \bm W) \leq \alpha^{(k)}
        \text{ for } k \in
        \mathcal{K}(\text{an}(\mathcal{S})) \mid \bm Z \in
        \mathcal{S}, \bm W\right\} \prod_{j \in
        \mathcal{K}(\{\mathcal{S}\} \cup \text{de}(\mathcal{S}))}
      \alpha_{j}.
    \end{split}
  \end{equation}
\end{lemma}

\subsection{Proof of \Cref{lemma:disjoint-or-nested}}
\label{sec:proof-crefl-or}

\begin{proof}
  Consider $\mathcal{S}\in
  \mathcal{R}^{(j)}$ and $\mathcal{S}'\in \mathcal{R}^{(k)}$ for some
  $j,k\in [K]$ and $\mathcal{S} \cap \mathcal{S}' \neq \emptyset$. By
  the definition of refinement and \cref{eq:nested},
  \[
  \mathcal{S}\cap
  \mathcal{S}'
  \in \underline{R}^{\{j,k\}} \subseteq \mathcal{R}^{(j)}\cup
  \mathcal{R}^{(k)},
  \]
so there exists
  integers $m$ such that $\mathcal{S}\cap \mathcal{S}' =
  \mathcal{S}^{(j)}_m$ or $\mathcal{S}^{(k)}_{m}$. Because
  $\mathcal{R}^{(j)}$ and $\mathcal{R}^{(k)}$ are partitions, this
  means that $\mathcal{S} = \mathcal{S}^{(j)}_m$ or $\mathcal{S}' =
  \mathcal{S}^{(k)}_{m}$. In either case, $\mathcal{S} \cap
  \mathcal{S}' = \mathcal{S}$ or $\mathcal{S}'$.

  Now consider any $\bm z\in \mathcal{Z}$ and $j,k \in [K]$. Because
  $\mathcal{S}_{\bm z}^{(j)},\mathcal{S}_{\bm z}^{(k)} \in
  \mathcal{R}^{[K]}$ and $\mathcal{S}_{\bm z}^{(j)} \cap
  \mathcal{S}_{\bm z}^{(k)} \neq \emptyset$ (because they contain at
  least $\bm z$), either $\mathcal{S}_{\bm z}^{(j)}  \supseteq
  \mathcal{S}_{\bm z}^{(k)} \text{or } \mathcal{S}_{\bm z}^{(k)}
  \supseteq \mathcal{S}_{\bm z}^{(j)}$ must be true. We can order the conditioning events $\mathcal{S}_{\bm z}^{(k)},
  k\in [K]$ according to the relation $\supseteq$. Without loss of
  generality, we assume that, at $\bm z$,
\[  \mathcal{S}_{\bm z}^{(1)}\supseteq \mathcal{S}_{\bm z}^{(2)}\supseteq \cdots\supseteq  \mathcal{S}_{\bm z}^{(K-1)}\supseteq \mathcal{S}_{\bm z}^{(K)}.
\]
  Then $\bigcap_{k=1}^{K}\mathcal{S}_{\bm z}^{(k)} =\mathcal{S}_{\bm
    z}^{(K)}$ and $\bigcup_{k=1}^{K}\mathcal{S}_{\bm z}^{(k)} =
  \mathcal{S}_{\bm z}^{(1)}$. Thus the intersection and the union of
  $\{\mathcal{S}_{\bm z}^{(k)}\}_{k=1}^K$ are contained in
  $\mathcal{R}^{[K]} $ which collects all $\mathcal{S}_{\bm z}^{(k)}$
  by the definition. As this is true for all $\bm z \in \mathcal{Z}$,
  we have $\underline{\mathcal{R}}^{[K]} \subseteq
  \mathcal{R}^{[K]}$ and $\overline{\mathcal{R}}^{[K]}\subseteq
  \mathcal{R}^{[K]}$.
\end{proof}

\subsection{Proof of \Cref{lem:hasse-structure}}
\label{sec:proof-crefl-struct}

\begin{proof}

  (i) Suppose $\mathcal{S}'$, $\mathcal{S}''$ are two distinct nodes
  in $\text{ch}(\mathcal{S})$. By \Cref{lemma:disjoint-or-nested},
  they are either disjoint or nested. If they are nested, without loss
  of generality, suppose $\mathcal{S}'' \supset \mathcal{S}'$. However,
  this contradicts with the definition of the edge $\mathcal{S}
  \rightarrow \mathcal{S}'$, as by \Cref{def:hasse} there should be no
  $\mathcal{S}''$ satisfying $\mathcal{S} \supset \mathcal{S}''\supset
  \mathcal{S}'$. This shows that any two nodes in
  $\text{ch}(\mathcal{S})$ are disjoint.

  Next we show that the union of the sets in $\text{ch}(\mathcal{S})$ is
  $\mathcal{S}$. Suppose there exists $\bm z \in \mathcal{S}$ such
  that $\bm{z} \not \in \mathcal{S}'$ for all $\mathcal{S}' \in
  \text{ch}(\mathcal{S})$. In consequence, $\bm{z} \not \in
  \mathcal{S}'$ for all $\mathcal{S}' \in
  \text{de}(\mathcal{S})$. Similar to the
  proof of \Cref{lemma:disjoint-or-nested}, we can order
  $\mathcal{S}_{\bm z}^{(k)}, k\in [K]$ according to set
  inclusion. Without loss of generality, suppose
\[
    \mathcal{S}_{\bm z}^{(1)}\supseteq \mathcal{S}_{\bm
      z}^{(2)}\supseteq \cdots\supseteq  \mathcal{S}_{\bm
      z}^{(K-1)}\supseteq \mathcal{S}_{\bm z}^{(K)}.
\]
  This shows that $\mathcal{S} = \mathcal{S}_{\bm
    z}^{(K)}$, so $\text{ch}(\mathcal{S})=
  \text{de}(\mathcal{S})=\varnothing$. This contradicts the
  assumption.

  (ii)
  Consider any $\mathcal{S}\in \mathcal{R}^{[K]},$ $\mathcal{S}' \in
  \text{an}(\mathcal{S})$ and $\mathcal{S}''\in
  \text{de}(\mathcal{S})$. By definition, $\mathcal{S}' \supset
  \mathcal{S} \supset \mathcal{S}''$. Because the sets in any
  partition $\mathcal{R}^{(k)}$ are disjoint, this shows that no pairs
  of $\mathcal{S},\mathcal{S}',\mathcal{S}''$ can belong to the same
  partition $\mathcal{R}^{(k)}$. Thus, $\mathcal{K}(\mathcal{S}'),
  \mathcal{K}(\mathcal{S}), \mathcal{K}(\mathcal{S}'')$ are
  disjoint. Because this is true for any $\mathcal{S}' \in
  \text{an}(\mathcal{S})$ and $\mathcal{S}''\in
  \text{de}(\mathcal{S})$, this shows
  $\mathcal{K}(\text{an}(\mathcal{S}))$, $\mathcal{K}(\mathcal{S})$,
  and $\mathcal{K}(\text{de}(\mathcal{S}))$ are disjoint.

By the proof of (i), for any $\bm z\in \mathcal{S}$,
  $\mathcal{S}$ is in a nested sequence consisting of
  $\mathcal{S}_{\bm z}^{(1)}, \dotsc,\mathcal{S}_{\bm z}^{(K)}$. Hence,
\[\mathcal{K}(\text{an}(\mathcal{S})) \cup \mathcal{K}(\mathcal{S})
  \cup \mathcal{K}(\text{de}(\mathcal{S})) =
  \mathcal{K}(\text{an}(\mathcal{S})\cup \{\mathcal{S}\}\cup
  \text{de}(\mathcal{S})) = [K].
\]

  (iii) The first result follows from the fact that
  $\text{an}(\mathcal{S}') = \text{an}(\mathcal{S}) \cup
  \{\mathcal{S}\}$. The second
  result is true because both $\mathcal{K}
  (\{\mathcal{S}'\}\cup  \text{de}(\mathcal{S}'))$ and $\mathcal{K} (
  \text{de}(\mathcal{S}))$ are equal to $[K] \setminus
  \mathcal{K}(\text{an}(\mathcal{S}) \cup \{\mathcal{S}\})$.
\end{proof}

\subsection{Proof of \Cref{lem:p-value-cond-indep}}
\label{sec:proof-crefl-value}

\begin{proof}
  First, we claim that for any $j,k \in [K]$ and $\bm z \in
  \mathcal{Z}$, given that $\bm Z \in
  \mathcal{S}^{(j)}_{\bm z} \cap \mathcal{S}^{(k)}_{\bm z}$,
  $P^{(k)}(\bm Z, \bm W)$ only depends on $\bm Z$ through $T^{(k)}(\bm Z, \bm W)$. This is true because of the
  definition of the p-value \cref{eq:p-value} and the invariance of
  the conditioning sets, i.e., $\mathcal{S}_{\bm Z}^{(k)} =
  \mathcal{S}^{(k)}_{\bm z}$.
  Now fix $\mathcal{S}\in \mathcal{R}^{[K]}$ (which means $\mathcal{S}
  = \mathcal{S}_{\bm z}$ for some $\bm z$) and consider any $j\in
  \mathcal{K}(\mathcal{S})$. Let $\mathcal{J} =
  \mathcal{K}(\text{an}(\mathcal{S})\cup \{\mathcal{S}\})$. By
  \Cref{lemma:disjoint-or-nested}, $\mathcal{S}_{\bm z}^{(j)}\cap
  \mathcal{S}_{\bm z}^{(k)} = \mathcal{S}_{\bm z}^{(j)} =
  \mathcal{S}$ for any $k \in \mathcal{J}$. Thus, the conditional
  independence \cref{eq:cond-indep-3} implies that
  \[
    t^{(j)}(\bm Z, \bm W) \independent \left( t^{(k)}(\bm Z, \bm W)
    \right)_{k \in \mathcal{J} \setminus \{j\}} \mid
    \bm Z \in \mathcal{S}, \bm W.
  \]
 Using the claim in the previous paragraph, this verifies
  the conclusion in \Cref{lem:p-value-cond-indep}.
  \end{proof}

\subsection{Proof of \Cref{lemma:general_lemma}}
\label{sec:proof-crefl}

\begin{proof}
  Let $\psi^{(k)} = \psi^{(k)}(\bm Z, \bm W) = 1_{\{P^{(k)}(\bm Z, \bm
    W) \leq \alpha^{(k)}\}}$ for $ k \in[K]$ and rewrite the L.H.S of
  \cref{equ:general_lemma}:
  \[
    \mathbb{P}\left\{P^{(1)}(\bm Z, \bm W) \leq \alpha^{(1)},\dotsc, P^{(K)}(\bm
      Z, \bm W) \leq \alpha^{(K)} \mid \bm Z \in \mathcal{S}, \bm
      W\right\} = \mathbb{E}\Bigg\{ \prod_{k=1}^{K}  \psi^{(k)}\mid
    \bm Z\in \mathcal{S}, \bm W \Bigg\}.
  \]
  We prove \Cref{lemma:general_lemma} by induction.
  First, consider any leaf node in the Hasse diagram, that is, any
  $\mathcal{S} \in \mathcal{R}^{[K]}$ such that
  $\text{ch}(\mathcal{S})=\emptyset$. By \Cref{lem:hasse-structure},
  $\mathcal{K}(\mathcal{S}) \cup \mathcal{K}(\text{an}(\mathcal{S})) =
  [K]$. By \Cref{lem:p-value-cond-indep},
  \[
  \psi^{(j)}
  \independent \left( \psi^{(k)}  \right)_{k \in
    \mathcal{K}(\mathcal{S}) \setminus \{j\}},~\forall j \in
  \mathcal{K}(\mathcal{S}).
  \]
Using
  the validity of each CRT (\Cref{thm:valid}),
  \begin{align*}
    \mathbb{E}\Bigg\{ \prod_{k=1}^{K}  \psi^{(k)}\mid  \bm Z\in \mathcal{S}, \bm W \Bigg\}
    &  =  \mathbb{E}\left\{\prod_{k\in
      \mathcal{K}(\text{an}(\mathcal{S}))} \psi^{(k)}\mid \bm Z\in
      \mathcal{S}, \bm W\right\} \prod_{j \in
      \mathcal{K}(\mathcal{S})} \mathbb{E}\left\{ \psi^{(j)}\mid \bm
      Z\in \mathcal{S}, \bm W\right\} \\
    & \leq  \mathbb{E}\left\{  \prod_{k\in
      \mathcal{K}(\text{an}(\mathcal{S}))} \psi^{(k)}\mid \bm Z\in
      \mathcal{S}, \bm W\right\} \prod_{j \in
      \mathcal{K}(\mathcal{S})} \alpha^{(j)}.
  \end{align*}
  This is exactly \cref{equ:general_lemma} for a leaf node. Now
  consider a non-leaf node $\mathcal{S} \in \mathcal{R}^{[K]}$
  (so $\text{ch}(\mathcal{S}) \neq \emptyset$) and suppose
  \cref{equ:general_lemma} holds for any descendant of
  $\mathcal{S}$. We have
  \begin{align*}
    &\  \mathbb{E}\Bigg\{ \prod_{k=1}^{K}  \psi^{(k)}\mid  \bm Z\in
      \mathcal{S}, \bm W \Bigg\}   \\
    = &\   \mathbb{E}\left\{ \sum_{\mathcal{S}'\in
        \text{ch}(\mathcal{S})}1_{\{\bm Z\in \mathcal{S}'\}}
        \mathbb{E}\left\{ \prod_{k=1}^{K}  \psi^{(k)}\mid
       \bm{Z}\in \mathcal{S}'   , \bm W  \right\} \mid \bm Z\in
        \mathcal{S}, \bm W\right\} \qquad \quad \text{(By
        \Cref{lem:hasse-structure}\ref{equ:hasse1})} \\
    \leq &\   \mathbb{E}\left\{\sum_{\mathcal{S}'\in
           \text{ch}(\mathcal{S})}1_{\{\bm Z\in \mathcal{S}'\}}
           \mathbb{E}\left\{ \prod_{k\in
           \mathcal{K}(\text{an}(\mathcal{S}'))} \psi^{(k)} \mid
          \bm{Z}\in \mathcal{S}'        , \bm W  \right\}
           \prod_{j\in
           \mathcal{K}(\{\mathcal{S}'\}\cup\text{de}(\mathcal{S}'))}
           \alpha^{(j)} \mid \bm Z\in \mathcal{S}, \bm W\right\} \\
    & \omit \hfill \text{(By the induction hypothesis)}  \\
    = &\   \mathbb{E}\left\{ \prod_{k\in
        \mathcal{K}(\text{an}(\mathcal{S})\cup \{\mathcal{S}\})}
        \psi^{(k)}   \mid \bm Z\in \mathcal{S}, \bm W\right\}
        \prod_{j\in \mathcal{K}( \text{de}(\mathcal{S}))} \alpha^{(j)}
        \qquad \qquad \qquad \quad
        \text{(By \Cref{lem:hasse-structure}\ref{equ:hasse3})}  \\
    =&\ \mathbb{E}\left\{ \prod_{k\in  \mathcal{K}(\text{an}(\mathcal{S}))}  \psi^{(k)}\mid \bm Z\in \mathcal{S}, \bm W\right\}
       \prod_{j\in  \mathcal{K}(\mathcal{S})} \mathbb{E}\left\{
       \psi^{(j)}\mid \bm Z\in \mathcal{S}, \bm W\right\}\prod_{j\in
       \mathcal{K}( \text{de}(\mathcal{S}))}  \alpha^{(j)} \\
    & \omit \hfill \text{(By \Cref{lem:hasse-structure}\ref{equ:hasse2} and
      \Cref{lem:p-value-cond-indep})}  \\
    \leq  &\ \mathbb{E}\left\{ \prod_{k\in
            \mathcal{K}(\text{an}(\mathcal{S}))}  \psi^{(k)}\mid \bm
            Z\in \mathcal{S}, \bm W\right\} \prod_{j\in
            \mathcal{K}(\{\mathcal{S}\} \cup \text{de}(\mathcal{S}))
            } \alpha^{(j)}. \qquad \qquad \qquad \qquad \text{(By \Cref{thm:valid})}
  \end{align*}
  By induction, this shows that \cref{equ:general_lemma} holds for all
  $\mathcal{S} \in \mathcal{R}^{[K]}$.
\end{proof}

\subsection{Proof of \Cref{thm:general-3}}
\label{sec:proof-crefthm:g-3}

\begin{proof}
  \Cref{lemma:disjoint-or-nested} shows that
  $\overline{\mathcal{R}}^{[K]} \subseteq \mathcal{R}^{[K]}$, thus
  \cref{equ:general_lemma} holds for every
  $\mathcal{S}\in \overline{\mathcal{R}}^{[K]} \subseteq
  \mathcal{R}^{[K]}$. Moreover, any $\mathcal{S}\in
  \overline{\mathcal{R}}^{[K]} $ has no superset in
  $\mathcal{R}^{[K]}$ and thus has no ancestors in the
  Hasse diagram. This means that the right-hand side of
  \cref{equ:general_lemma} is simply
  $\prod_{k=1}^{K}\alpha^{(k)}$. Because $\overline{\mathcal{G}}^{[K]}$ is
  the $\sigma$-algebra generated by the sets in $\overline{\mathcal{R}}^{[K]}$,
  equation \cref{equ:dominance-conditional-3} holds. Finally,
  equation \cref{equ:dominance-3} holds trivially by taking the
  expectation of \cref{equ:dominance-conditional-3} over
  $\overline{\mathcal{G}}^{[K]}$.
\end{proof}

\subsection{Proof of \Cref{thm:combine} in \Cref{sec:combine}}\label{proof:combine}

We denote the units used in the $k$-th test by $\mathcal{I}^{(k)} = \mathcal{I}_0^{(k)}\cup \mathcal{I}_1^{(k)}$. We denote the treatment variables and outcomes used in the $k$-th test  by
$\bm A^{(k)} = (A_i : i \in \mathcal{I}^{(k)}) $ and $\bm Y^{(k)} = (Y_{i}^{(k)}: i \in \mathcal{I}^{(k)})$, respectively. We define the difference-in-means statistics in the $k$-th test as
\begin{equation}\label{equ:dims}
t(\bm A^{(k)}, \bm Y^{(k)}) =\sqrt{N} (\bar{Y}_{1}^{(k)} - \bar{Y}_{0}^{(k)}),
\end{equation}
where the average outcomes are given by
\[
\bar{Y}_{1}^{(k)} = (N_{1}^{(k)})^{-1}\sum_{i\in \mathcal{I}^{(k)}}A_{i}^{(k)}Y_{ik}  \quad \text{and}\quad  \bar{Y}_{0}^{(k)} = (N_{0}^{(k)})^{-1}\sum_{i\in \mathcal{I}^{(k)}}(1-A_{i}^{(k)})Y_{ik}.
\]
Let $\bm W^{(k)} = \big(Y_{i}^{(k)}(0), Y_{i}^{(k)}(1):i\in \mathcal{I}^{(k)}   \big)$. The randomization distribution in the $k$-th test is
\[
  \hat{G}^{(k)}(b^{(k)}) = \mathbb{P}^*\left\{t\big(\bm A_{*}^{(k)},\bm Y^{(k)}(\bm A_{*}^{(k)})\big) \leq b^{(k)} \mid  \bm W^{(k)}  \right\},
  \]
where $\bm A_{*}^{(k)}$ is a permutation of  $\bm A^{(k)}$, i.e., $\bm A_{*}^{(k)}$ is drawn from the same uniform assignment distribution  as $\bm A^{(k)}$. Under assumption (i), using the bivariate Central Limit Theorem,
\[
\Big(t\big(\bm A^{(k)}, \bm Y^{(k)}\big), t\big(\bm A_{*}^{(k)},\bm Y^{(k)}(\bm A_{*}^{(k)}) \big)\Big)\xrightarrow{d} \big(T_{\infty}^{(k)},T_{\infty,*}^{(k)} \big),
\]
where $T_{\infty}^{(k)}$ and $T_{\infty,*}^{(k)}$ are independent, each with a common c.d.f $G^{(k)}(\cdot).$ By \citet[Theorem 15.2.3]{lehmann2006testing},
$ \hat{G}^{(k)}(b^{(k)}) \xrightarrow{P} G^{(k)}(b^{(k)})$
for every $b^{(k)}$ which is a continuity point of $G^{(k)}(\cdot)$.

Under assumptions (i), (ii) and the zero-effect null hypothesis
$H_0^{(k)}$, using the result in \citet[Theorem 15.2.5]{lehmann2006testing}, we have
\begin{equation}\label{equ:need_proof}
T_{\infty}^{(k)} \sim g_{0}^{(k)} = \mathcal{N}(0,V_{\infty}^{(k)}),
\end{equation}
where the variance $V_{\infty}^{(k)}$ takes the form
\begin{equation}\label{equ:need_proof}
    V_{\infty}^{(k)} = 1/\Lambda^{(k)}=\frac{N}{N_0^{(k)}} \Var\big[Y^{(k)}(1)\big] + \frac{N}{N_1^{(k)}} \Var\big[Y^{(k)}(0)\big].
\end{equation}
This expression of the asymptotic variance can be found in
\citet[Section 6.4]{imbens2015causal}. For completeness, an
alternative derivation is provided in \Cref{proof:lehmann}.
Under assumptions  (i), (ii) and the constant effect alternative hypothesis
$H_{1}^{(k)}$ (with $\tau = h/\sqrt{N}$),
\begin{equation}
T_{\infty}^{(k)} \sim g_{1}^{(k)} = \mathcal{N}( h,V_{\infty}^{(k)}).
\end{equation}
The c.d.f of $T_{\infty}^{(k)}$ under $H_{0}^{(k)}$ and $H_{1}^{(k)}$ are given by
\[
G_{0}^{(k)}(b^{(k)}) = \Phi\left(b^{(k)}/\sqrt{V_{\infty}^{(k)}}\right)\quad \text{and} \quad G_{1}^{(k)}(b^{(k)}) = \Phi\left(\Big[b^{(k)}-h\Big]/\sqrt{V_{\infty}^{(k)}} \right).
\]
Let $\tilde{b}^{(k)}= \big[ G_{0}^{(k)}\big]^{-1}(p^{(k)})$. The p-value density function under $H_{1}^{(k)}$ can be rewritten as
\[
  f_{1}^{(k)}(p^{(k)}) = g_{1}^{(k)}\big(\tilde{b}^{(k)}\big)\left|\frac{d \tilde{b}^{(k)}}{d p^{(k)}} \right| = g_{1}^{(k)}(\tilde{b}^{(k)})\left|\bigg( \big[G_{0}^{(k)}\big]'(\tilde{b}^{(k)})\bigg)^{-1}\right| = g_{1}^{(k)}(\tilde{b}^{(k)})/ g_{0}^{(k)}(\tilde{b}^{(k)}).
\]
Since the p-values from the permutation tests follow a standard uniform distribution under the null,
the log-likelihood ratio of the p-values $P^{(1)}\dotsc, P^{(K)}$ is given by
\[
\sum_{k=1}^{K} \log \left[ f_{1}^{(k)}\big(p^{(k)}\big) \right] = \sum_{k=1}^{K} \log \left[ \frac{g_{1}^{(k)}(\tilde{b}^{(k)})}{ g_{0}^{(k)}(\tilde{b}^{(k)})} \right] \propto \sum_{k=1}^{K} \sqrt{\Lambda^{(k)}} \Phi^{-1}(p^{(k)}).
\]
Using the p-value weights $\Lambda^{(k)},k\in[K],$ from the log-likelihood ratio, we obtain
\begin{equation}\label{equ:weighted_t}
T_\infty = \sum_{k=1}^{K} w_{\infty}^{(k)} \Phi^{-1}(P^{(k)}) \quad\text{where}\quad w_\infty^{(k)}  = \sqrt{\frac{\Lambda^{(k)}}{\sum_{j=1}^{K}\Lambda^{(j)}} }.
\end{equation}

\subsection{Proof of \cref{equ:need_proof} in \Cref{proof:combine}}\label{proof:lehmann}
\citet[Theorem 15.2.3]{lehmann2006testing} uses the fact that under assumptions (i) and (ii), $V_{\infty}^{(k)}= \Var[T\big(\bm A^{(k)}, \bm Y^{(k)}\big)]$, but leaves the  calculation of $\Var[T\big(\bm A^{(k)}, \bm Y^{(k)}\big)]$ (their Equation 15.15) as an exercise for readers. Here we provide the calculation to complete our proof for
\cref{equ:need_proof}.
To simplify the exposition, we let $m= N_1^{(k)}$, $n=N_0^{(k)}$.
Suppose that $\mathcal{I}^{(k)} = [m+n]$,
$\mathcal{I}_1^{(k)} = [m], \mathcal{I}_0^{(k)} = \{m+1,\dotsc,m+n\}$ and
\[
\bm Y^{(k)} = (\underbrace{Y_{1}^{(k)},\cdots,Y_{m}^{(k)}}_\text{treated outcomes},\underbrace{Y_{m+1}^{(k)},\dotsc,Y_{m+n}^{(k)}}_\text{control outcomes}).
\]
Let $\Pi = \big[\Pi(1),\dotsc,\Pi(m+n) \big]$ be an independent random permutation of $1,\dotsc,m+n.$
We rewrite the difference-in-means statistics \cref{equ:dims} as
\[
T := t\big(\bm A^{(k)}, \bm Y^{(k)}\big) =   \frac{\sqrt{N}}{m} \sum_{i=1}^{m+n}E_i Y_{i}^{(k)},\text{ where }
E_i =\begin{cases}
        1 & \text{if }   \Pi(i)\leq m \\
        -m/n & \text{otherwise}
\end{cases}
, \forall i\in [m+n].
\]
Let $D$ be the number of $i\leq m$ such that $\Pi(i)\leq m $. The value of $T$ is determined by $m-D$, the number of units swapped between the treated group ($i=1,\dotsc m$) and control group ($i=m+1,\dotsc, m+n$). Since all the treated (control) units have the same outcome variance, it does not  matter which units are swapped in computing the variance of $T$.
We first consider the case that $m\leq n$. Using the expectation of the hypergeometric distribution,
\[
\mathbb E[D] = \sum_{d=0}^{m}\frac{{m \choose d}{n \choose m-d          }}{             {m+n \choose m          }               } d = \frac{m^2}{m+n}.
\]
Let $\sigma_1^2 = \Var[Y^{(k)}(1)]$ and $\sigma_0^2 = \Var[Y^{(k)}(0)].$ The variance of $T$ is given by
\begin{align*}
\Var(T)& =\frac{N}{m^2}\sum_{d=0}^{m}\frac{{m \choose d}{n \choose m-d          }}{             {m+n \choose m}}\left[  d\sigma_1^2 + (m-d)\sigma_0^2+(m-d) \frac{m^2}{n^2}\sigma_1^2   + (n-m+d)               \frac{m^2}{n^2} \sigma_0^2              \right] \\
&=\frac{N}{m^2}\left[   \frac{m^2}{m+n}\sigma_1^2 + \frac{mn}{m+n}\sigma_0^2+\frac{mn}{m+n} \frac{m^2}{n^2}\sigma_1^2   + \frac{m^2}{m+n}               \sigma_0^2              \right]\\
&=\frac{N}{m^2}\left[   \frac{m^2}{n}\sigma_1^2 + m\sigma_0^2   \right]\\
&=      \frac{N}{n}\sigma_1^2 + \frac{N}{m} \sigma_0^2.
\end{align*}
If $m>n$, we let $D$ denote the number of $i\in \{m+1,\dotsc,m+n\}$ such that $\Pi(i) \in \{m+1,\dotsc,m+n\}$. Then,
$\mathbb E[D] = \frac{n^2}{m+n}$, \text{ and }
\begin{align*}
\Var(T)& =\frac{N}{m^2}\sum_{d=0}^{n}\frac{{m \choose d}{n \choose m-d          }}{             {m+n \choose n}}\left[  d \frac{m^2}{n^2}\sigma_0^2 + (n-d)\frac{m^2}{n^2} \sigma_1^2+(n-d) \sigma_0^2  + (m-n+d)               \sigma_1^2              \right] \\
&=\frac{N}{m^2}\left[   \frac{m^2}{m+n}\sigma_0^2 + \frac{mn}{m+n}\frac{m^2}{n^2}\sigma_1^2+\frac{mn}{m+n} \sigma_0^2   + \frac{m^2}{m+n}               \sigma_1^2              \right]\\
&=\frac{N}{m^2}\left[   \frac{m^2}{n}\sigma_1^2 + m\sigma_0^2   \right]\\
&=      \frac{N}{n}\sigma_1^2 + \frac{N}{m} \sigma_0^2.
\end{align*}
The variance is the same for $m\leq n$ and $m>n$. This proves our claim and Equation 15.15 in \citet{lehmann2006testing}: $\Var(m^{-1/2} \sum_{i=1}^{m+n}E_iY_{i}^{(k)} ) = \Var(\sqrt{\frac{m}{N}}T)= \frac{m}{n}\sigma_1^2 +  \sigma_0^2.$

\subsection{Proof of \Cref{lemma:p_hat} in \Cref{sec:combine}}\label{proof:p_hat}
In \Cref{alg:mcrts}, the time steps in $\mathcal{C}_j$ define a sequence of nested permutation tests. For example, $\mathcal{C}_1 = \{1,3,5,7\}$
defines CRTs 1,3,5 and $\mathcal{C}_2 = \{2,4,6,8\}$ defines CRTs 2,4,6 in
 \Cref{fig:swd-crt}. Suppose that we only have one subset $\mathcal{C}=\{c_1,\dotsc,c_{K} \}$ consists of $K$ time points.

 We next prove that the tests defined on $\mathcal{C}$ are valid to combine. It is suffices to show that the test statistics  $\hat{T} = \sum_{k=1}^{K} \hat{w}^{(k)} \Phi^{-1}(P^{(k)}) $ stochastically dominates the random variable
\[
\tilde{T} :=  \sum_{k=1}^{K} \hat{w}^{(k)} \Phi^{-1}(U^{(k)}) \sim \mathcal{N}\left(0,\sum_{k=1}^{K} \big[\hat{w}^{(k)} \big]^{2}\right) = \mathcal{N}(0,1).
\]
where each $U^{(k)}$ is a standard uniform. The second equality is achieved by the definition of $\hat{w}^{(k)},k\in [K].$

By conditioning on the potential outcomes $(\bm W^{(k)},k\in [K])$, the weights $\hat{w}^{(k)},k\in [K],$ are fixed. The derivation follows the same steps as  \Cref{prop:nested} (the conditioning on $\bm W$'s is suppressed). By construction, $c_1<\dotsc<c_{K}$ and the conditioning sets $\mathcal{S}^{(1)}\supseteq \cdots \supseteq \mathcal{S}^{(K)}$ for all $\bm z\in \mathcal{Z}$. Then, the corresponding $\sigma$-algebras generated by the conditioning sets satisfy that
$ \mathcal{G}^{(1)} \subseteq \cdots \subseteq \mathcal{G}^{(K)}$. Conditioning on $\mathcal{G}^{(k')}$, the term $\sum_{k=1}^{k'-1}\hat{w}^{(k)}\Phi^{-1}(P^{(k)})$ is fixed. By the law of iterated expectation,
\begin{align*}
 \mathbb{E}\left[               1\left\{\hat{T}\leq b   \right\}    \right]=\  & \mathbb{E}\left[1\left\{ \hat{w}^{(K)}  \Phi^{-1}(P^{(K)})\leq b -\sum_{k=1}^{K-1}\hat{w}^{(k)}\Phi^{-1}(P^{(k)})     \right\} \right] \\
   =\ &  \mathbb{E}\left( \mathbb{E}\left[1{\left\{ \hat{w}^{(K)} \Phi^{-1}(P^{(K)})\leq b -\sum_{k=1}^{K-1}\hat{w}^{(k)}\Phi^{-1}(P^{(k)})   \right\}} \ \Big|\ \mathcal{G}^{(K)} \right]\right)     \\
   \leq\ & \mathbb{E}_{U^{(K)}}\left( \mathbb{E}\left[1 \left\{ \hat{w}^{(K)}  \Phi^{-1}(U^{(K)})\leq b -\sum_{k=1}^{K-1}\hat{w}^{(k)}\Phi^{-1}(P^{(k)})       \right\} \ \Big|\ U^{(K)}  \right]   \right)    \\
 =\ & \mathbb{E}_{U^{(K)}}\bigg( \mathbb{E}\bigg[1 \bigg\{ \hat{w}^{(K-1)} \Phi^{-1}(P^{(K-1)}) \leq b -\sum_{k=1}^{K-2}\hat{w}^{(k)}\Phi^{-1}(P^{(k)})  \\
 &  \hspace{150pt} - \hat{w}^{(K)} \Phi^{-1}(U^{(K)})     \bigg\} \ \Big|\ \mathcal{G}^{(K-1)}, U^{(K)}  \bigg]  \bigg)   \\
 \leq\ & \mathbb{E}_{U^{(K-1)},U^{(K)}}\bigg( \mathbb{E}\bigg[1 \bigg\{ \hat{w}^{(K-1)} \Phi^{-1}(U^{(K-1)}) \leq b -\sum_{k=1}^{K-2}\hat{w}^{(k)}\Phi^{-1}(P^{(k)}) \\
& \hspace{150 pt} - \hat{w}^{(K)} \Phi^{-1}(U^{(K)})      \bigg\} \ \Big|\ U^{(K-1)},U^{(K)}  \bigg]  \bigg)      \\
 &  \hspace{-10pt} \vdots \\
\leq \ & \mathbb{E}_{U^{(1)},\dotsc,U^{(K)}}\bigg( \mathbb{E}\bigg[1 \bigg\{ \sum_{k=1}^{K}\hat{w}^{(k)}\Phi^{-1}(U^{(k)}) \leq b  \bigg\} \ \Big|\ U^{(1)},\dotsc,U^{(K)}  \bigg]  \bigg)\\
     =\  & \mathbb{E}\left[ 1{\left\{ \tilde{T}  \leq b \right\}} \right].
\end{align*}
The inequalities are attained by the validity of the p-values $P^{(1)},\dotsc,P^{(K)}$, i.e., each $P^{(k)}$ stochastically dominates the standard uniform variable $U^{(k)}$. Since $\hat{T}$ stochastically dominates the standard normal random variable $\tilde T$,
\begin{equation}\label{equ:combin_p}
\mathbb P\{\hat P \leq \alpha  \} = \mathbb P\{\Phi(\hat T) \leq \alpha  \}  \leq \mathbb P\{\Phi(\tilde T) \leq \alpha  \}= \alpha
\end{equation}
The last equality is achieved by the fact that $\Phi(\tilde T)$ is a standard uniform random variable.
Suppose that \Cref{alg:mcrts} creates multiple subsets $\mathcal{C}_j,j\in [J].$ Applying the same proof to the tests defined on each $\mathcal{C}_j$,
\[
 \mathbb{E}\left[               1\left\{\hat{T}_j\leq b \right\}    \right] \leq  \mathbb{E}\left[              1\left\{\tilde{T}_j\leq b       \right\}    \right].
\]
where $\hat{T}_j=\sum_{c\in \mathcal{C}_j}\hat{w}^{(c)} \Phi(P^{(c)})$ and
$\tilde{T}_j=\sum_{c\in \mathcal{C}_j}\hat{w}^{(c)} \Phi(U^{(c)})\sim \mathcal{N}\left(0,\sum_{c\in \mathcal{C}_j} \big[\hat{w}^{(c)} \big]^{2}\right)$.
By splitting the time steps, we make sure that
 the p-values from the tests based on different $\mathcal{C}_j$ can be combined. The test statistics
 $\hat{T} = \sum_{j\in\mathcal J} \hat{T}_j $ stochastically dominates the random variable
 \[
 \tilde{T} = \sum_{j\in\mathcal J} \tilde{T}_j\sim \mathcal{N}\bigg(0,\sum_{j\in \mathcal{J}}\sum_{c\in \mathcal{C}_j} \big[\hat{w}^{(c)} \big]^{2} \bigg)
 =  \mathcal{N}(0,1),
 \]
 which implies that \cref{equ:combin_p} still holds for the combined p-value $\hat P  = \Phi( \hat{T})$.





\end{appendices}




\end{document}